\documentclass{amsart}
\usepackage{graphicx}
\usepackage[all]{xy}
\usepackage{amscd}
\usepackage{amssymb}
\usepackage{amsmath}
\usepackage{amsthm}
\usepackage{amsfonts}
\usepackage{graphics}
\usepackage{epsfig,psfrag}
\usepackage[english]{babel}
\usepackage{latexsym}
\usepackage{mathrsfs}

\newtheorem{theorem}{Theorem}[section]

\newtheorem{lemma}[theorem]{Lemma}

\newtheorem{proposition}[theorem]{Proposition}

\theoremstyle{definition}

\theoremstyle{remark}
\newtheorem{remark}[theorem]{Remark}

\begin{document}

\title[Minimal surfaces in $S^3$ with high symmetry]
{Minimal surfaces in the three dimensional sphere with high symmetry}

\author{Sheng Bai}
\address{School of Mathematical Sciences, Peking University, Beijing 100871, CHINA}
\email{barries@163.com}

\author{Chao Wang}
\address{Jonsvannsveien 87B, H0201, Trondheim 7050, NORWAY}
\email{chao\_{}wang\_{}1987@126.com}

\author{Shicheng Wang}
\address{School of Mathematical Sciences, Peking University, Beijing 100871, CHINA}
\email{wangsc@math.pku.edu.cn}



\thanks{This work was supported by National Natural Science Foundation of China (Grant Nos. 11371034 and 11501534).}

\begin{abstract}
Using the Lawson's existence theorem of minimal surfaces and the symmetries of the Hopf fibration, we will construct symmetric embedded closed minimal surfaces in the three dimensional sphere. These surfaces contain the Clifford torus, the Lawson's minimal surfaces, and seven new minimal surfaces with genera $9$, $25$, $49$, $121$, $121$, $361$ and $841$. We will also discuss the relation between such surfaces and the maximal extendable group actions on subsurfaces of the three dimensional sphere.
\end{abstract}

\maketitle

\section{Introduction}\label{Sec:intro}

In \cite{L1} Lawson established two kinds of reflection principles for minimal surfaces in the three dimensional sphere $S^3$. Using one of them he constructed a family of embedded closed minimal surfaces in $S^3$. Moreover, the genus $g$ of such a surface can be any positive integer, and the minimal surface is not unique if $g$ is not a prime number. Then based on Lawson's work, in \cite{KPS} Karcher, Pinkall and Sterling constructed nine new embedded closed minimal surfaces in $S^3$. Their construction relies on the symmetric tessellations of $S^3$, and both two reflection principles were used. Moreover, their examples disproved Lawson's equal volume conjecture, which says that any embedded closed minimal surface in $S^3$ separates $S^3$ into two components of equal volume. Note that in \cite{L3} Lawson proved that the two components must be diffeomorphic.

The main purpose of this paper is to give more minimal surfaces in $S^3$. We will construct seven new embedded closed minimal surfaces in $S^3$ with genera $9$, $25$, $49$, $121$, $121$, $361$ and $841$. Since our surfaces admit high symmetry, they also give new embedded closed minimal surfaces in classical spherical manifolds. For example, there exist three embedded closed minimal surfaces with genera $2$, $4$ and $8$ in the Poincar\'e's homology three sphere.

The construction of our surfaces will be similar to the construction in \cite{L1}. Namely we will find certain quadrilaterals in $S^3$ such that each edge of them is a geodesic and each angle of them has the form $\pi/(l+1)$ with $l$ a positive integer. Then by Lawson's existence theorem the quadrilaterals will bound minimal disks such that via successive $\pi$-rotations around the edges one can get closed minimal surfaces. However, to find the quadrilaterals we need well understanding of the Hopf fibration and the isometries of $S^3$ that preserve the fibration. The geodesic quadrilaterals will be the lifts of certain piecewise geodesics in $S^2$.

For the convenience of computation, in the paper we will consider the ``enlarged'' Hopf fibration given by the map $\mathcal{P} : S^3_2 \rightarrow S^2$, $r \mapsto r^{-1}ir$ as in Lemma \ref{Lem:lift}, which is from the three dimensional sphere with radius $2$ in the quaternion space $\mathbb{H}$ to the unit sphere in the $ijk$-space. The piecewise geodesics will be like the figure ``8''. Edge lengths and angles of their lifts are listed in Section \ref{ssec:fund}. Figure \ref{fig:HopfGQ} is a sketch map of this construction. We have mapped $S^3_2\setminus\{-2\}$ to the $ijk$-space by the map $t+xi+yj+zk \mapsto (xi+yj+zk)/(t+2)$, and red lines denote the fibres.

\begin{figure}[h]
\centerline{\includegraphics{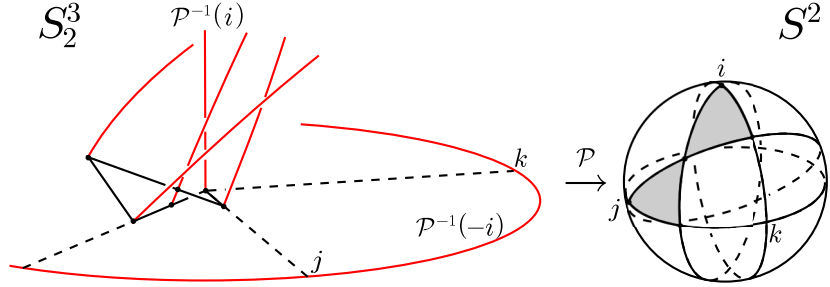}}
\caption{Hopf fibration and geodesic quadrilateral in $S^3_2$}\label{fig:HopfGQ}
\end{figure}

By this way we can get three families of quadrilaterals which correspond to the Clifford torus and Lawson's minimal surfaces, and eight special quadrilaterals which can give seven different new minimal surfaces. The isometric groups of the minimal surfaces can also be obtained (Proposition \ref{Pro:iso}), and we will see that the minimal surfaces admit high symmetries. Actually the motivation of this paper was inspired by the results about maximal extendable group actions on subsurfaces of $S^3$.

Let $\Sigma_g$ be a genus $g>1$ embedded smooth closed surface in $S^3$, and $G$ is a finite group acting smoothly and faithfully on the pair $(\Sigma_g, S^3)$, then the $G$-action on $\Sigma_g$ is called extendable over $S^3$ with respect to the embedding. In \cite{WWZZ} the authors showed that if elements in $G$ preserve both orientations of $\Sigma_g$ and $S^3$ and the order of $G$ is bigger than $4(g-1)$, then the orbifold pair $(\Sigma_g/G, S^3/G)$ can be classified. As a corollary, for a given $g$ the maximum of the order of $G$ was determined. We call the extendable $G$-action maximal if the order of $G$ reaches the maximum. By the results in \cite{WWZZ}, all the maximal extendable $G$-actions consist of two infinite sequences and some special cases as in Table \ref{tab:maximal action}.

\begin{table}[h]
\caption{Maximal extendable actions}\label{tab:maximal action}\centerline{
\begin{tabular}{|c|c|}
\hline order of $G$ & genus $g>1$ (count by multiplicity)\\\hline
 $12(g-1)$ & $\{3, 5, 6, 11, 17, 601\}_{KPS}, \{9, 11, 121, 241, 241\}_K,$\\
 & $9, 11, 25, 97, 121, 241$\\
 $8(g-1)$ & $\{7, 73\}_{KPS}, 49$\\
 $20(g-1)/3$ & $\{19\}_{KPS}, \{361\}_K, 361$\\
 $6(g-1)$ & $\{21, 481\}_K^{\infty}$\\
 $24(g-1)/5$ & $41$\\
 $30(g-1)/7$ & $\{29, 841\}_L, 29, 841, 1681$\\
 $4(l+1)^2$ & $\{l^2, l\neq3, 5, 7, 11, 19, 29, 41\}_L$\\
 $4(g+1)$ & $\{\text{remaining numbers}\}_L$\\
\hline
\end{tabular}}
\end{table}

In Table \ref{tab:maximal action}, except $21$ and $481$ each number in the right part represents a maximal extendable action. For example, in the $12(g-1)$ case the three $241$'s represent three non-conjugate actions. For $21$ and $481$ there are infinitely many maximal extendable actions. The subscript $K$ means that at least one side of the embedded surface is not a handlebody. Such a surface (or the embedding) is usually called knotted.


For a minimal surface $M$ in $S^3$, let $Isom(M)$ be the group generated by the isometries of $S^3$ which keep $M$ invariant, and let $Isom^+(M)$ be the group generated by the elements of $Isom(M)$ which preserve both orientations of $M$ and $S^3$. Then it happens that most maximal extendable actions are conjugate to some $Isom^+(M)$-action on $(M, S^3)$ where $M$ is some embedded closed minimal surface constructed in \cite{L1} or \cite{KPS}. In Table \ref{tab:maximal action}, the subscripts $L$ and $KPS$ indicate this correspondence. Note that in \cite{L3} Lawson proved that each side of an embedded closed minimal surface must be a handlebody. Hence for actions with subscript $K$ the embedded surfaces can not be minimal.

In \cite{WWZ} the authors determined the maximum order of general extendable $G$-actions, namely elements in $G$ may reverse the orientation of $M$ or $S^3$. For the general maximal extendable actions we can have Table \ref{tab:gen max act}, and in this case it can be shown that the embedded surface can not be knotted.

\begin{table}[h]
\caption{General maximal extendable actions}\label{tab:gen max act}
\centerline{
\begin{tabular}{|c|c|}
\hline order of $G$ & genus $g>1$ (count by multiplicity)\\\hline
 $48(g-1)$ & $\{6\}_{KPS}$\\
 $32(g-1)$ & $\{73\}_{KPS}$\\
 $24(g-1)$ & $\{5\}_L, \{5, 11, 17, 601\}_{KPS}, 11 ,25, 97, 121$\\
 $16(l+1)^2$ & $\{l^2, l\neq 11\}_L$\\
 $16(g+1)$ & $\{\text{remaining numbers}\}_L$\\
\hline
\end{tabular}}
\end{table}

Our results actually partly answers the following question (Proposition \ref{Pro:max}).

{\bf Question :} Is every general maximal extendable action (maximal extendable action on unknotted $\Sigma_g$) conjugate to some $Isom(M)$-action ($Isom^+(M)$-action) on $(M, S^3)$ such that $M$ is minimal?

In section \ref{Sec:basic}, we will give some preliminary lemmas about spherical geometry and Hopf fibration, which are basic, see \cite{CS}. We will mainly use Lemma \ref{Lem:lift} and Lemma \ref{Lem:trans} in the constructions of minimal surfaces, and readers who are familiar with these contents can skip this section.

In section \ref{Sec:construction}, we will give the constructions of minimal surfaces. We will lift the skeletons of classical tessellations of $S^2$ to get skeletons of the minimal surfaces. Then we will show that the skeletons in $S^3_2$ are unions of geodesic quadrilaterals. Explicit information about the quadrilaterals and the corresponding groups of the skeletons will be given.

In section \ref{Sec:existence}, we will introduce the Lawson's existence theorem. Then we will give a version of the maximal principle and a uniqueness lemma about minimal surfaces. Finally, we will use the lemmas to verify the conditions in the Lawson's existence theorem and show the existence of the minimal surfaces.

In section \ref{Sec:symmetry}, we will determine the isometric groups of the minimal surfaces. Then we will consider some orbifold pairs related to the maximal extendable actions and give some further discussions about the results in \cite{WWZZ} and \cite{WWZ}.

\section{On spherical geometry and Hopf fibration}\label{Sec:basic}

\subsection{Quaternion and orthogonal actions}

Let $\mathbb{R}$, $\mathbb{C}$ and $\mathbb{H}$ denote the real number, the complex number and the quaternion number respectively. Then
\begin{align*}
\mathbb{C}&=\{t+xi \mid t, x \in \mathbb{R}\},\\
\mathbb{H}&=\{t+xi+yj+zk \mid t, x, y, z \in \mathbb{R}\}\\
          &=\{z_1+z_2j \mid z_1, z_2 \in \mathbb{C}\},
\end{align*}
where $i^2=j^2=k^2=-1$, $ij=-ji=k$, $jk=-kj=i$, $ki=-ik=j$.

For $z=t+xi \in \mathbb{C}$, let $\bar{z}=t-xi$, $|z|=\sqrt{t^2+x^2}$, then $\bar{\bar{z}}=z$, $z\bar{z}=|z|^2$ and $jz=\bar{z}j$. For $z_1+z_2j, w_1+w_2j \in \mathbb{H}$, we have
$$(z_1+z_2j)(w_1+w_2j)=(z_1w_1-z_2\bar{w}_2)+(z_1w_2+z_2\bar{w}_1)j.$$
For $h=z_1+z_2j \in \mathbb{H}$, let $\bar{h}=\bar{z}_1-z_2j$, $|h|=\sqrt{|z_1|^2+|z_2|^2}$, then $\bar{\bar{h}}=h$, $h\bar{h}=|h|^2$. Moreover, for $z_1, z_2, w_1, w_2 \in \mathbb{C}$, we have $\overline{z_1z_2}=\bar{z}_1\bar{z}_2$, $\overline{z_1\pm z_2}=\bar{z}_1\pm\bar{z}_2$. Hence
\begin{align*}
\overline{(z_1+z_2j)(w_1+w_2j)}&= (\bar{z}_1\bar{w}_1-\bar{z}_2w_2)-(z_1w_2+z_2\bar{w}_1)j\\
&=(\bar{w}_1-w_2j)(\bar{z}_1-z_2j)\\
&=(\overline{w_1+w_2j})(\overline{z_1+z_2j}).
\end{align*}
Hence for $p, q \in \mathbb{H}$ we have
$$|pq|^2=pq\overline{pq}=pq\bar{q}\bar{p}=|p|^2|q|^2.$$
Namely $|pq|=|p||q|$.

Identify $\mathbb{H}$ with the four dimensional Euclidean space. Then every element $h \in \mathbb{H}$ can be thought as a vector, and $|h|$ is the norm of $h$. Let
\begin{align*}
E^3&=\{t+xi+yj+zk \in \mathbb{H} \mid t=0\},\\
S^3&=\{t+xi+yj+zk \in \mathbb{H} \mid t^2+x^2+y^2+z^2=1\},\\
S^2&=\{xi+yj+zk \in E^3 \mid x^2+y^2+z^2=1\}=S^3\cap E^3.
\end{align*}
Then $E^3$ is a hyperplane which is orthogonal to $1$. We assume that $i$, $j$, $k$ form a right hand orthogonal system of $E^3$. The two sets $S^3$ and $S^2$ are unit spheres in $\mathbb{H}$ and $E^3$ respectively. Clearly $E^3$ has standard Euclidean geometry, $S^3$ and $S^2$ have standard spherical geometry.

Let $SO(4)$ be the orientation preserving isometric group of $\mathbb{H}$ which preserves $S^3$, and let $SO(3)$ be the orientation preserving isometric group of $E^3$ which preserves $S^2$. For $p, q \in S^3$, $h \in \mathbb{H}$, we have $p^{-1}=\bar{p}$ and
$$|p^{-1}hq|=|p^{-1}||h||q|=|h|.$$
Hence we get an element of $SO(4)$.
\begin{align*}
\mathbb{H} & \rightarrow \mathbb{H}\\
h & \mapsto p^{-1}hq
\end{align*}
Denote it by $[p, q]$. Then $[p, p]$ preserves $1 \in \mathbb{H}$. Hence it also preserves $E^3$, and it gives an element of $SO(3)$. Denote it by $[p]$. The map
\begin{align*}
\mathbb{H} & \rightarrow \mathbb{H}\\
h & \mapsto -\bar{h}
\end{align*}
is the reflection about $E^3$. Hence generally the reflection about the hyperplane, which passes $0 \in \mathbb{H}$ and is orthogonal to $p$, is given by the following map.
\begin{align*}
\mathbb{H} & \rightarrow \mathbb{H}\\
h & \mapsto -\overline{hp^{-1}}p
\end{align*}
Notice that every element of $SO(4)$ can be presented as a composition of even reflections about hyperplanes passing $0 \in \mathbb{H}$. We have the following lemma.

\begin{lemma}\label{Lem:2to1}
The following two maps are surjective homomorphisms, with kernels generated by $(-1, -1)$ and $-1$ respectively.
\begin{align*}
S^3\times S^3 & \rightarrow SO(4) & S^3 & \rightarrow SO(3)\\
(p, q) & \mapsto [p,q] & p & \mapsto [p]
\end{align*}
\end{lemma}

The action of $SO(4)$ on $\mathbb{H}$ is a right action. Namely, for $[p_1, q_1], [p_2, q_2] \in SO(4)$, $h \in \mathbb{H}$, we have
$$(h^{[p_1, q_1]})^{[p_2, q_2]} =h^{[p_1, q_1][p_2, q_2]}
=h^{[p_1p_2, q_1q_2]},$$
or equivalently
$$p_2^{-1}(p_1^{-1}hq_1)q_2 =(p_2^{-1}p_1^{-1})h(q_1q_2)
=(p_1p_2)^{-1}h(q_1q_2).$$
Hence the action of $SO(3)$ on $E^3$ is also a right action.

\subsection{The identification of $SO(3)$ and $UTS^2$}

Let $UTS^2$ be the unit tangent bundle of $S^2$, which consists of all unit tangent vectors of $S^2$. An element in $UTS^2$ can be presented as $(P, v) \in E^3\times E^3$, where $P$ is a point in $S^2$, and $v$ is a unit tangent vector at $P$. The projection map
\begin{align*}
UTS^2 & \rightarrow S^2\\
(P,v) & \mapsto P
\end{align*}
gives a fibration. The fibre at $P$ is the unit circle in the tangent space at $P$.

Let $\mathbb{C}P^1$ be the one dimensional complex projective space. For $z_1+z_2j \in S^3$, let $[z_1:z_2]$ denote the ratio of $z_1$ and $z_2$. Then the map
\begin{align*}
S^3 & \rightarrow \mathbb{C}P^1\\
z_1+z_2j & \mapsto [z_1:z_2]
\end{align*}
gives the Hopf fibration. Each fibre has the form $\{e^{i\theta}p \mid \theta \in \mathbb{R}\}$, where $p \in S^3$. The antipodal map of $S^3$ preserves the Hopf fibration, and on each fibre it is also an antipodal map. By Lemma \ref{Lem:2to1}, it induces a fibration on $SO(3)$. Each fibre of the induced fibration has the form $\{[e^{i\theta}p] \mid \theta \in \mathbb{R}\}$, where $p \in S^3$.

We choose the unit tangent vector $j$ at the point $i \in S^2$. Then $(i,j) \in UTS^2$.

\begin{lemma}\label{Lem:identify}
The following map $\mathcal{I}$ gives an identification of $SO(3)$ and $UTS^2$. It also gives an identification of the fibration of $SO(3)$, which is induced from the Hopf fibration, and the fibration of $UTS^2$.
\begin{align*}
\mathcal{I} : SO(3) & \rightarrow UTS^2\\
[p] & \mapsto (p^{-1}ip, p^{-1}jp)
\end{align*}
Moreover, the $\theta$-increasing direction of the fibre $\{[e^{i\theta}p] \mid \theta\in \mathbb{R}\}$ corresponds to the direction of left hand rotations around $p^{-1}ip$.
\end{lemma}

\begin{proof}
Since the action of $SO(3)$ on $UTS^2$ is faithful and transitive, the map $\mathcal{I}$ is a bijection. Consider the image of a fibre $\{[e^{i\theta}p] \mid \theta\in \mathbb{R}\}$ in $SO(3)$. We have
\begin{align*}
(e^{i\theta}p)^{-1}i(e^{i\theta}p)&=p^{-1}ip,\\
(e^{i\theta}p)^{-1}j(e^{i\theta}p)&=p^{-1}(e^{-2i\theta}j)p\\
&=p^{-1}(\cos(-2\theta)j+\sin(-2\theta)k)p.
\end{align*}
Hence $\mathcal{I}$ maps the fibre $\{[e^{i\theta}p] \mid \theta\in \mathbb{R}\}$ in $SO(3)$ to the fibre at $p^{-1}ip$ in $UTS^2$, and the increasing of $\theta$ corresponds to left hand rotations.
\end{proof}

As a byproduct, for $z_1+z_2j \in S^3$, the following map induced by $\mathcal{I}$ gives an identification of $\mathbb{C}P^1$ and $S^2$.
\begin{align*}
\mathbb{C}P^1 & \rightarrow S^2\\
[z_1:z_2] & \mapsto (z_1+z_2j)^{-1}i(z_1+z_2j)
\end{align*}

On $UTS^2$ we use the sub-topology from $E^3\times E^3$, and on $SO(3)$ we use the quotient-topology from $S^3$. Then $\mathcal{I}$ is a homeomorphism.

\subsection{The metrics on $SO(3)$ and $UTS^2$}

Let ``$\langle\cdot ,\cdot\rangle$'' denote the inner product in $E^3$. Given $\epsilon>0$, let $(P(s),v(s))$, $s \in (-\epsilon,\epsilon)$, be a smooth curve in $UTS^2$. Then we have the tangent vector $(dP(s)/ds,dv(s)/ds) \in E^3\times E^3$. Notice that
$$\langle P(s), P(s)\rangle=1, \langle v(s), v(s)\rangle=1, \langle P(s), v(s)\rangle=0.$$
We have
\begin{align*}
&\langle\frac{dP(s)}{ds}, P(s)\rangle=0, \langle\frac{dv(s)}{ds}, v(s)\rangle=0,\\
&\langle\frac{dP(s)}{ds}, v(s)\rangle+\langle P(s), \frac{dv(s)}{ds}\rangle=0.
\end{align*}
Let $f(s)$ be the cross product $P(s)\times v(s)$, then
\begin{align*}
\frac{dP(s)}{ds}&= \langle\frac{dP(s)}{ds}, v(s)\rangle v(s) + \langle\frac{dP(s)}{ds}, f(s)\rangle f(s),\\
\frac{dv(s)}{ds}&=-\langle\frac{dP(s)}{ds}, v(s)\rangle P(s) + \langle\frac{dv(s)}{ds}, f(s)\rangle f(s).
\end{align*}
Hence the tangent space of $UTS^2$ at the point $(P,v)$ is generated by
$$(v,-P), (P\times v,0), (0,P\times v).$$
As $(P,v)$ varies in $UTS^2$, we get three tangent vector fields. Under the standard Euclidean metric of $E^3\times E^3$ the three vectors at $(P,v)$ are orthogonal to each other, and their norms are $\sqrt{2}$, $1$ and $1$ respectively. We define a new metric on the tangent space such that the three vectors form a standard orthogonal system, namely the norm of each vector is $1$. Then we get a metric on $UTS^2$.

Let $S^3_2$ be the three dimensional sphere with radius $2$ in $\mathbb{H}$. For $r \in S^3_2$, we have $r/2 \in S^3$. Hence we have the following two to one covering map.
\begin{align*}
S^3_2 & \rightarrow SO(3)\\
r & \mapsto [r/2]
\end{align*}
Then the spherical metric on $S^3_2$ induces a metric on $SO(3)$ such that locally the covering map is an isometry.

\begin{lemma}\label{Lem:metric}
With above metrics on $SO(3)$ and $UTS^2$, $\mathcal{I}$ is an isometry.
\end{lemma}
\begin{proof}
Notice that the right action of $SO(3)$ on $UTS^2$ is an isometric group action, and via right multiplication there is an isometric group action of $SO(3)$ on itself. Moreover, the two actions are conjugate via the map $\mathcal{I}$. Hence we only need to consider the tangent map at $[1] \in SO(3)$.

Consider the map
$$S^3_2\rightarrow SO(3)\rightarrow UTS^2.$$
Let
\begin{align*}
&2(\cos(s/2)+i\sin(s/2)), s \in \mathbb{R},\\
&2(\cos(s/2)+j\sin(s/2)), s \in \mathbb{R},\\
&2(\cos(s/2)+k\sin(s/2)), s \in \mathbb{R},
\end{align*}
be the three great circles containing $2$ in $S^3_2$. At the point $2$ they have tangent vectors $i$, $j$ and $k$ respectively. Via the images of the three circles in $UTS^2$, we can compute the tangent map, and $i$, $j$, $k$ are mapped to $(0,-k)$, $(k,0)$ and $(-j,i)$ at the point $(i,j) \in UTS^2$. Then by definitions of the metrics on $SO(3)$ and $UTS^2$, the tangent map at $[1]$ is an isometry.
\end{proof}

\subsection{Parallel translation and Gauss-Bonnet formula}

Given $\epsilon>0$, let $P(s)$, $s \in (-\epsilon,\epsilon)$, be a smooth curve in $S^2$, and let $v$ be a unit tangent vector at $P(0)$. Then the ordinary differential equation
$$\frac{dv(s)}{ds}=-\langle\frac{dP(s)}{ds}, v(s)\rangle P(s)$$
always has a unique smooth solution $v(s)$ such that $v(0)=v$. Since
$$\langle\frac{dv(s)}{ds}, P(s)\rangle=
-\langle\frac{dP(s)}{ds}, v(s)\rangle\langle P(s), P(s)\rangle= -\langle\frac{dP(s)}{ds}, v(s)\rangle$$
and $\langle P(0), v(0)\rangle=0$, we have $\langle P(s), v(s)\rangle=0$. Then since
$$\langle\frac{dv(s)}{ds}, v(s)\rangle=
-\langle\frac{dP(s)}{ds}, v(s)\rangle\langle P(s), v(s)\rangle=0$$
and $\langle v(0), v(0)\rangle=1$, we have $\langle v(s), v(s)\rangle=1$. Hence we get a smooth unit tangent vector field $v(s)$, which gives the parallel translation of $v$ along $P(s)$. We also get a smooth curve $(P(s),v(s))$ in $UTS^2$, which is the lift of $P(s)$ passing $(P(0),v)$.

If $P(s)$ is a geodesic in $S^2$, then its unit tangent vector field gives a parallel translation, and other parallel translations along $P(s)$ can be obtained by, at each point $P(s)$, rotating the unit tangent vector around the vector $P(s)$ by the same angle. For a piecewise geodesic $L(s)$, $s \in [0,1]$, and a unit tangent vector $v$ at $L(0)$, we can get a parallel translation of $v$ along $L(s)$ piecewise, and obtain a unit tangent vector $v'$ at $L(1)$.

Let $L(s)$, $s \in [0,1]$, be a simple closed piecewise geodesic in a closed hemisphere of $S^2$. Since $L(1)=L(0)$, $v$ and $v'$ are in the same tangent space. $L(s)$ bounds a subsurface $S$ in the hemisphere. Looking from outside of $S^2$ and moving along $L(s)$, $S$ is at the left side or the right side of $L(s)$. If $S$ is at the left (right) side of $L(s)$, then the vector $v'$ can be obtained from $v$ by a right (left) hand rotation with an angle $\Theta(L(s)) \in (0,2\pi]$. Let $Area(S)$ be the area of $S$. We have the following Gauss-Bonnet formula in this simple case.

\begin{lemma}\label{Lem:GBthm}
Under above assumptions, $\Theta(L(s))=Area(S)$.
\end{lemma}

\begin{proof}
If $L(s)$ is a geodesic triangle in some hemisphere of $S^2$, then the formula is the same as the area formula of the geodesic triangle in $S^2$.

In general case, $L(s)$ is a geodesic polygon. We can assume that there is a geodesic $\gamma$ in $S$ from $L(0)$ to $L(s_0)$, $0<s_0<1$. Suppose that it divides $S$ into $S_1$ and $S_2$, it divides $L(s)$ into $\alpha$ and $\beta$, and the boundaries of $S_1$ and $S_2$ are $\alpha\gamma^{-1}$ and $\gamma\beta$ respectively. Then since $Area(S)\leq 2\pi$, if
$$\Theta(\alpha\gamma^{-1})=Area(S_1),\Theta(\gamma\beta)=Area(S_2),$$ we must have $\Theta(L(s))=Area(S)$. Hence we can divide $S$ into geodesic triangles, and get the result by induction.
\end{proof}

\subsection{Geometric properties of $S^3_2$}

Given $\epsilon>0$, let $P(s)$, $s \in (-\epsilon,\epsilon)$, be a smooth curve in $S^2$, and let $v$ be a unit tangent vector at $P(0)$. Let $(P(s),v(s))$ be the lift of $P(s)$ passing $(P(0),v)$ in $UTS^2$, and let $f(s)$ be the cross product of $P(s)$ and $v(s)$. By the definition of the metric on $UTS^2$, the tangent subspace at the point $(P(s),v(s))$, which is orthogonal to $(0,f(s))$, is isometric to the tangent space at the point $P(s)$.

Since $dv(s)/ds$ is orthogonal to $f(s)$, in $UTS^2$ the lift $(P(s),v(s))$ intersects the fibres orthogonally. Then locally the lift $(P(s),v(s))$ has the same length as the curve $P(s)$. Hence for a piecewise geodesic $L(s)$ in $S^2$, its lift $(L(s),v(s))$ in $UTS^2$ is also a piecewise geodesic, which is locally isometric to $L(s)$. It intersects the fibres orthogonally and it has the same angles as $L(s)$ at the corner points.

Let $L(s)$, $s \in [0,1]$, be a simple closed piecewise geodesic which bounds a subsurface $S$ in a closed hemisphere of $S^2$. Let $(L(s),v(s))$ be a lift of $L(s)$ from $(L(0),v)$ to $(L(1),v')=(L(0),v')$. By Lemma \ref{Lem:identify}, \ref{Lem:metric} and \ref{Lem:GBthm}, suppose that $L(0)=p^{-1}ip$ for some $p \in S^3$, then a left (right) hand rotation around the fibre at $L(0)$ in $UTS^2$ with the angle $\Theta(L(s)) \in (0,2\pi]$ is the same as a movement along the $\theta$-increasing ($\theta$-decreasing) direction of the fibre $\{[e^{i\theta}p] \mid \theta\in \mathbb{R}\}$ in $SO(3)$ with distance $\Theta(L(s))$, which is equal to $Area(S)$. Since the covering map from $S^3_2$ to $SO(3)$ (or $UTS^2$) preserves local metric. We have the following lemma.

\begin{lemma}\label{Lem:lift}
The following map $\mathcal{P}$ from the three dimensional sphere of radius $2$ in $\mathbb{H}$ to the unit sphere in $E^3$, which is generated by $i$, $j$, $k$ in $\mathbb{H}$, gives an enlarged Hopf fibration.
\begin{align*}
\mathcal{P} : S^3_2 & \rightarrow S^2\\
r & \mapsto r^{-1}ir
\end{align*}
Moreover, with the spherical geometries on $S^3_2$ and $S^2$ we have

(a) Each fibre has the form $\{e^{i\theta/2}r \mid \theta \in \mathbb{R}\}$ with $r \in S^3_2$, and has length $4\pi$.

(b) At $r \in S^3_2$, the tangent subspace which is orthogonal to the fibre is isometric to the tangent space at $\mathcal{P}(r)$.

(c) Given $r \in S^3_2$, each piecewise geodesic $L(s)$, $s \in [0,1]$, in $S^2$ with $L(0)=\mathcal{P}(r)$ can be uniquely lifted to a piecewise geodesic $\widetilde{L}(s)$ in $S^3_2$ with $\widetilde{L}(0)=r$, such that

(c1) $\widetilde{L}(s)$ intersects the fibres orthogonally.

(c2) $\widetilde{L}(s)$ is locally isometric to $L(s)$, and $\widetilde{L}(s)$ has the same angles as $L(s)$ at the corner points.

(d) If the $L(s)$ in (c) is simple and closed, and it bounds a subsurface $S$ in a closed hemisphere of $S^2$, then $\widetilde{L}(1)=e^{-iArea(S)/2}\widetilde{L}(0)$ or $\widetilde{L}(1)=e^{iArea(S)/2}\widetilde{L}(0)$, depending on $S$ is at the left side or the right side of $L(s)$.
\end{lemma}

\subsection{Orthogonal actions preserving the Hopf fibration}

Let $\mathcal{P} : S^3_2\rightarrow S^2$ be the enlarged Hopf fibration. Each fibre has the form $\{e^{i\theta/2}r \mid \theta \in \mathbb{R}\}$ with $r \in S^3_2$. Let $[p,q]$ be an element in $SO(4)$, where $p,q \in S^3$.

\begin{lemma}\label{Lem:FPaction}
If $[p,q]$ preserves the fibration given by $\mathcal{P}$, then $p=e^{i\tau}$ or $p=e^{i\tau}j$ for some $\tau \in \mathbb{R}$, and via $\mathcal{P}$ it induces an isometry on $S^2$, which is given by $[q]$ or the composition of $[q]$ and the antipodal map of $S^2$.
\end{lemma}

\begin{proof}
Since $[p,q]=[1,q][p,1]$ and $[1,q]$ preserves the fibration, we only need to consider $[p,1]$. Let $p^{-1}=z_1+z_2j$, then for the fibre $\{2e^{i\theta/2} \mid \theta \in \mathbb{R}\}$, we have
$$p^{-1}(2e^{i\theta/2})=2e^{i\theta/2}z_1+2e^{-i\theta/2}z_2j.$$
If $[p,1]$ preserves the fibration, then $[e^{i\theta/2}z_1:e^{-i\theta/2}z_2]$ is a constant. Hence $z_2=0$ or $z_1=0$, namely $p=e^{i\tau}$ or $p=e^{i\tau}j$ for some $\tau \in \mathbb{R}$. Then for $r \in S^3_2$, we have
$$\mathcal{P}(r^{[p,q]})= (p^{-1}rq)^{-1}i(p^{-1}rq)= (r^{-1}(pip^{-1})r)^{[q]}.$$
Hence if $p$ is $e^{i\tau}$ or $e^{i\tau}j$, then $\mathcal{P}(r^{[p,q]})$ is equal to $\mathcal{P}(r)^{[q]}$ or $-\mathcal{P}(r)^{[q]}$.
\end{proof}

\begin{lemma}\label{Lem:rotate}
Given $q \in S^3$, then it can be written as $q=\cos(\phi/2)+u\sin(\phi/2)$,
where $\phi \in \mathbb{R}$ and $u \in S^2$. Then $[q]$ is the left hand rotation around $u$ with angle $\phi$, and for $r \in \mathcal{P}^{-1}(u)$, we have $r^{[1,q]}=e^{i\phi/2}r$.
\end{lemma}

\begin{proof}
There exists $p \in S^3$ such that $p^{-1}ip=u$. Then $p^{-1}e^{i\phi/2}p=q$. Since $[e^{i\phi/2}]$ is a left hand rotation around $i$ with angle $\phi$ and $u^{[p^{-1}]}=i$, $[q]=[p^{-1}][e^{i\phi/2}][p]$ must be the left hand rotation around $u$ with angle $\phi$.

If $r \in \mathcal{P}^{-1}(u)$, then $ir=ru$. Hence $r^{[1,q]}=r\cos(\phi/2)+ru\sin(\phi/2)=e^{i\phi/2}r$.
\end{proof}

\begin{lemma}\label{Lem:geodesic}
The geodesic passing $2 \in S^3_2$ and intersecting the fibres orthogonally has the form $2(\cos(s/2)+e^{i\tau}j\sin(s/2))$, $s \in \mathbb{R}$, where $\tau$ is some fixed real number. Then the $\pi$-rotation around the geodesic has the form $[e^{i\tau}j,e^{i\tau}j]$.
\end{lemma}

\begin{proof}
The geodesic passing $2 \in S^3_2$ has the form $(2e^{is/2})^{[p,p]}$, $s\in \mathbb{R}$, where $p \in S^3$. Its tangent vector at $2$ is $i^{[p]} \in S^2$. If it is orthogonal to the fibre $\{2e^{i\theta/2} \mid \theta\in \mathbb{R}\}$, then $i^{[p]}$ is orthogonal to $i$. Hence $i^{[p]}=e^{i\tau}j$ for some $\tau \in \mathbb{R}$, and $(2e^{is/2})^{[p,p]}$ has the required form. Since $[1,(e^{i\theta/2})^{[p,p]}]$, where $\theta \in \mathbb{R}$, preserves the fibration and the geodesic, the geodesic must intersect the fibration orthogonally.

The isometry $[e^{i\tau}j,e^{i\tau}j]$ fixes points in the geodesic, and for $\theta\in \mathbb{R}$ it maps $2e^{i\theta/2}$ to $2e^{-i\theta/2}$, hence it is a $\pi$-rotation around the geodesic.
\end{proof}

As a byproduct of the proof, every geodesic in $S^3_2$ intersects the fibres with the same angle. From Lemma \ref{Lem:FPaction}, \ref{Lem:rotate} and \ref{Lem:geodesic} we have the following lemma.

\begin{lemma}\label{Lem:trans}
Orthogonal actions preserving the enlarged Hopf fibration of $S^3_2$ are generated by the following elements.

(a) $[e^{i\tau/2},1]$, where $\tau \in \mathbb{R}$, is a movement of distance $\tau$ along the $\theta$-decreasing direction of all fibres $\{e^{i\theta/2}r \mid \theta\in \mathbb{R}\}$ in $S^3_2$. It induces the identity map on $S^2$.

(b) $[1,\cos(\phi/2)+u\sin(\phi/2)]$, where $\phi \in \mathbb{R}$ and $u \in S^2$, induces the left hand rotation around $u$ with angle $\phi$ on $S^2$. On the fibre $\mathcal{P}^{-1}(u)$, it is a movement along the $\theta$-increasing direction with distance $\phi$.

(c) $[e^{i\tau}j,e^{i\tau}j]$, where $\tau \in \mathbb{R}$, is the $\pi$-rotation around the geodesic
$$2(\cos(s/2)+e^{i\tau}j\sin(s/2)), s \in \mathbb{R}.$$
On $S^2$, it induces a reflection about the plane which is orthogonal to $e^{i\tau}j$.
\end{lemma}

\section{The construction of minimal surfaces}\label{Sec:construction}

\subsection{Finite reflection groups of $S^2$}

Let $u$ be an element of $S^2$. By Lemma \ref{Lem:rotate} the $\pi$-rotation around $u$ is given by $[u]$. For $p \in S^2$, its image under the antipodal map of $S^2$ is $-p=\bar{p}=p^{-1}$. Hence we can use $-[u]$ to present the reflection about the plane in $E^3$ which passes $0 \in E^3$ and is orthogonal to $u$.

In the isometric group of $S^2$, there are five classes of finite subgroups which are generated by reflections. They can be obtained from the five classes of finite subgroups of $SO(3)$ by adding suitable reflections. We use symbols {\bf C}, {\bf D}, {\bf T}, {\bf O}, {\bf I} to denote them, and list their generators as below, where $l$ is some positive integer.
\begin{align*}
{\bf C} &: [\cos\frac{\pi}{l}+i\sin\frac{\pi}{l}], -[j] \\
{\bf D} &: [\cos\frac{\pi}{l}+i\sin\frac{\pi}{l}], [k], -[j]\\
{\bf T} &: [\cos\frac{\pi}{2}+i\sin\frac{\pi}{2}], [\cos\frac{\pi}{3}+\frac{i+j+k}{\sqrt{3}}\sin\frac{\pi}{3}], -[e^{i\pi/4}k]\\
{\bf O} &: [\cos\frac{\pi}{4}+i\sin\frac{\pi}{4}], [\cos\frac{\pi}{3}+\frac{i+j+k}{\sqrt{3}}\sin\frac{\pi}{3}], -[k]\\
{\bf I} &: [\cos\frac{\pi}{2}+i\sin\frac{\pi}{2}], [\cos\frac{\pi}{3}+\frac{(\sqrt{5}+1)i+(\sqrt{5}-1)j}{2\sqrt{3}}
\sin\frac{\pi}{3}], -[k]
\end{align*}

For a finite reflection group of $S^2$, all of its reflection circles form a graph. The graph cuts $S^2$ into congruent pieces, which gives $S^2$ a tessellation. Let $\Gamma_C$, $\Gamma_D$, $\Gamma_T$, $\Gamma_O$, $\Gamma_I$ denote the graphs corresponding to {\bf C}, {\bf D}, {\bf T}, {\bf O}, {\bf I} respectively. Let $u_T$, $u_T'$, $u_O$, $u_I$, $u_I'$ be the following points in $S^2$.
\begin{align*}
u_T&=\frac{i+j+k}{\sqrt{3}}, u_T'=\frac{i+j-k}{\sqrt{3}}, u_O=\frac{i+j}{\sqrt{2}}, \\
u_I&=\frac{(\sqrt{5}+1)i+(\sqrt{5}-1)j}{2\sqrt{3}},\\
u_I'&=\frac{(\sqrt{10+2\sqrt{5}})i-(\sqrt{10-2\sqrt{5}})k}{2\sqrt{5}}.
\end{align*}
Then for each of the tessellations given by $\Gamma_C$, $\Gamma_D$, $\Gamma_T$, $\Gamma_O$, $\Gamma_I$, we can have a typical piece, which is denoted by $\triangle_C$, $\triangle_D$, $\triangle_T$, $\triangle_O$, $\triangle_I$ respectively.

\begin{figure}[h]
\centerline{\includegraphics{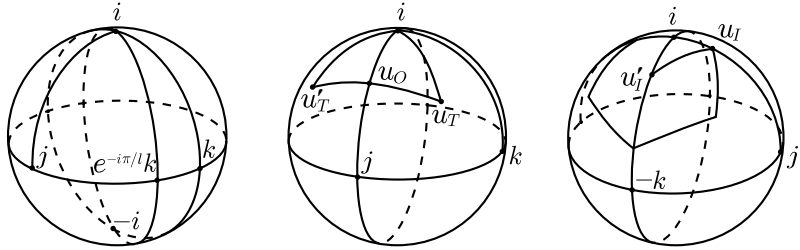}}
\caption{Typical pieces}\label{fig:TPiece}
\end{figure}

In Figure \ref{fig:TPiece}, we marked the vertices of the typical pieces. The concrete description of $\triangle_C$, $\triangle_D$, $\triangle_T$, $\triangle_O$, $\triangle_I$ is as below.

$\triangle_C$ : a geodesic bi-gon with vertices $i$, $-i$, and edges passing $k$, $e^{-i\pi/l}k$.

$\triangle_D$ : a geodesic triangle with angles $\pi/l$, $\pi/2$, $\pi/2$ at vertices $i$, $k$, $e^{-i\pi/l}k$.

$\triangle_T$ : a geodesic triangle with angles $\pi/2$, $\pi/3$, $\pi/3$ at vertices $i$, $u_T$, $u_T'$.

$\triangle_O$ : a geodesic triangle with angles $\pi/4$, $\pi/3$, $\pi/2$ at vertices $i$, $u_T$, $u_O$.

$\triangle_I$ : a geodesic triangle with angles $\pi/2$, $\pi/3$, $\pi/5$ at vertices $i$, $u_I$, $u_I'$.

Note that the areas of the pieces are $2\pi/l$, $\pi/l$, $\pi/6$, $\pi/12$, $\pi/30$ respectively.

\subsection{The skeletons of minimal surfaces}\label{ssec:skel}

Since each of the graphs $\Gamma_C$, $\Gamma_D$, $\Gamma_T$, $\Gamma_O$, $\Gamma_I$ passes the point $i \in S^2$, by Lemma \ref{Lem:lift}, we can lift them piecewise from the point $2 \in S^3_2$. With properly choosing of lifts, we will get graphs in $S^3_2$, which are unions of great circles and will be the skeletons of minimal surfaces. For each lifted graph, we will give a subgroup of $SO(4)$ preserving it and the enlarged Hopf fibration of $S^3_2$. These groups are generated by $\pi$-rotations around the great circles in the graphs, and can be thought as the lifts of the reflection groups of $S^2$.

\subsubsection{The lifts of $\Gamma_C$ and $\Gamma_D$}

(The skeletons of Lawson's minimal surfaces.)

The lifts of $\Gamma_C$ : Let $m$ and $n$ be two positive integers, and let $l$ be the lowest common multiple of $m$ and $n$. Firstly, we lift geodesics in $\Gamma_C$ passing the points $i$ and $e^{i\pi t/m}k$, $t=0, 1, \cdots, m-1$. In $S^3_2$, we get $m$ great circles passing $2$, which intersect the fibre $\mathcal{P}^{-1}(-i)$. Let $\widetilde{\alpha}$ be such a great circle in $S^3_2$ intersecting $\mathcal{P}^{-1}(-i)$ at a point $r'$, and let $\alpha=\mathcal{P}(\widetilde{\alpha})$, which passes some point $e^{i\pi t_0/m}k$. Then by Lemma \ref{Lem:lift}, from the point $r' \in S^3$ we can lift the geodesics in $\Gamma_C$ passing the points $-i$ and $e^{i\pi t'/n}e^{i\pi t_0/m}k$, $t'=0, 1, \cdots, n-1$.

This procedure can be done successively. When a lift $\widetilde{\alpha}$ meets the fibre $\mathcal{P}^{-1}(i)$ at some point $r$, then from $r$ we lift geodesics in $\Gamma_C$ having angles $\pi t/m$ with $\alpha$; when a lift $\widetilde{\alpha}$ meets the fibre $\mathcal{P}^{-1}(-i)$ at some point $r'$, then from $r'$ we lift geodesics in $\Gamma_C$ having angles $\pi t'/n$ with $\alpha$. Finally, we can get a graph in $S^3_2$ which is the union of great circles. Denote this graph by $\widetilde{\Gamma}_C(m,n)$.

The lifts of $\Gamma_D$ : There are two kinds of lifts, $\widetilde{\Gamma}_D(l)$ and $\widetilde{\Gamma}_D(l,\pi/2)$.

$\widetilde{\Gamma}_D(l)$: Lift the geodesic in $\Gamma_D$ passing the points $i$ and $k$. In $S^3_2$, we get a great circle passing $2$, which intersects the fibres $\mathcal{P}^{-1}(\pm k)$. Then from the intersection points, we lift all geodesics in $\Gamma_D$ passing $\pm k$. Generally, when a lift $\widetilde{\alpha}$ meets the fibre $\mathcal{P}^{-1}(e^{i\pi t/l}k)$, $t=0,1,\cdots,2l-1$, at some point $r$, then from $r$ we lift all geodesics in $\Gamma_D$ passing $e^{i\pi t/l}k$. Finally, we can get the graph in $S^3_2$.

$\widetilde{\Gamma}_D(l,\pi/2)$: Let $V_D(\pi/2)$ be the orbit of the vertex $e^{-i\pi/l}k$, which has an angle $\pi/2$ in the typical piece $\triangle_D$, under the action of {\bf D}. Lift all geodesics in $\Gamma_D$ passing the point $i$. In $S^3_2$, we get great circles passing $2$. Then lift $\Gamma_D$ successively. When a lift $\widetilde{\alpha}$ meets the fibre $\mathcal{P}^{-1}(u)$ at some point $r$, and $u$ is a vertex of $\Gamma_D$ not in $V_D(\pi/2)$, then from $r$ lift all geodesics in $\Gamma_D$ passing $u$. Finally, we can get the graph in $S^3_2$. Note that $\widetilde{\Gamma}_D(l,\pi/2)$ contains $\widetilde{\Gamma}_C(l,l)$ as a subgraph.

\subsubsection{The lifts of $\Gamma_T$, $\Gamma_O$, $\Gamma_I$}

(The skeletons of special minimal surfaces.)

There are two lifts of $\Gamma_T$, $\widetilde{\Gamma}_T(\pi/2)$ and $\widetilde{\Gamma}_T(\pi/3)$; three lifts of $\Gamma_O$, $\widetilde{\Gamma}_O(\pi/2)$, $\widetilde{\Gamma}_O(\pi/3)$ and $\widetilde{\Gamma}_O(\pi/4)$; and three lifts of $\Gamma_I$, $\widetilde{\Gamma}_I(\pi/2)$, $\widetilde{\Gamma}_I(\pi/3)$ and $\widetilde{\Gamma}_I(\pi/5)$.

$\widetilde{\Gamma}_T(\pi/2)$: Let $V_T(\pi/2)$ be the orbit of the vertex $i$ under the action of {\bf T}. Lift the geodesic in $\Gamma_T$ passing the points $i$ and $u_T$. In $S^3_2$, we get a great circle passing $2$. Then lift $\Gamma_T$ successively. When a lift $\widetilde{\alpha}$ meets the fibre $\mathcal{P}^{-1}(u)$ at some point $r$, and $u$ is a vertex of $\Gamma_T$ not in $V_T(\pi/2)$, then from $r$ lift all geodesics in $\Gamma_T$ passing $u$. Finally, we can get the graph in $S^3_2$.

The lifts $\widetilde{\Gamma}_O(\pi/4)$, $\widetilde{\Gamma}_I(\pi/2)$ can be obtained similarly. For $\widetilde{\Gamma}_O(\pi/4)$, the symbols $V_T(\pi/2)$, {\bf T}, $\Gamma_T$, $u_T$ should be replaced by $V_O(\pi/4)$, {\bf O}, $\Gamma_O$, $u_O$; and for $\widetilde{\Gamma}_I(\pi/2)$, the symbols $V_T(\pi/2)$, {\bf T}, $\Gamma_T$, $u_T$ should be replaced by $V_I(\pi/2)$, {\bf I}, $\Gamma_I$, $u_I$.

$\widetilde{\Gamma}_T(\pi/3)$: Let $V_T(\pi/3)$ be the orbit of the vertex $u_T'$ under the action of {\bf T}. Lift all geodesics in $\Gamma_T$ passing the point $i$. In $S^3_2$, we get great circles passing $2$. Then lift $\Gamma_T$ successively. When a lift $\widetilde{\alpha}$ meets the fibre $\mathcal{P}^{-1}(u)$ at some point $r$, and $u$ is a vertex of $\Gamma_T$ not in $V_T(\pi/3)$, then from $r$ lift all geodesics in $\Gamma_T$ passing $u$. Finally, we can get the graph in $S^3_2$.

The remaining lifts can be obtained similarly. For $\widetilde{\Gamma}_O(\pi/2)$, the symbols $V_T(\pi/3)$, $u_T'$, {\bf T}, $\Gamma_T$ should be replaced by $V_O(\pi/2)$, $u_O$, {\bf O}, $\Gamma_O$; for $\widetilde{\Gamma}_O(\pi/3)$, the symbols $V_T(\pi/3)$, $u_T'$, {\bf T}, $\Gamma_T$ should be replaced by $V_O(\pi/3)$, $u_T$, {\bf O}, $\Gamma_O$; for $\widetilde{\Gamma}_I(\pi/3)$, the symbols $V_T(\pi/3)$, $u_T'$, {\bf T}, $\Gamma_T$ should be replaced by $V_I(\pi/3)$, $u_I$, {\bf I}, $\Gamma_I$; and for $\widetilde{\Gamma}_I(\pi/5)$, the symbols $V_T(\pi/3)$, $u_T'$, {\bf T}, $\Gamma_T$ should be replaced by $V_I(\pi/5)$, $u_I'$, {\bf I}, $\Gamma_I$.

\subsubsection{The lifts of reflection groups of $S^2$}\label{sssec:LofG}

For the three families of lifted graphs from $\Gamma_C$, $\Gamma_D$ and eight lifted graphs from $\Gamma_T$, $\Gamma_O$, $\Gamma_I$ we define their corresponding groups as below. The generators and orders of the groups are given.
\begin{align*}
G_C(m,n) &: [e^{i\pi/m},\cos\frac{\pi}{m}+i\sin\frac{\pi}{m}], [e^{i\pi/n},\cos\frac{\pi}{n}-i\sin\frac{\pi}{n}], [j,j]; &&2mn\\
G_D(l) &: [e^{i\pi/2},\cos\frac{\pi}{2}+k\sin\frac{\pi}{2}], [e^{i\pi/2},\cos\frac{\pi}{2}+e^{-i\pi/l}k\sin\frac{\pi}{2}], [j,j]; &&8l\\
G_D(l,\frac{\pi}{2}) &: [e^{i\pi/l},\cos\frac{\pi}{l}+i\sin\frac{\pi}{l}], [e^{i\pi/2},\cos\frac{\pi}{2}+k\sin\frac{\pi}{2}], [j,j]; &&4l^2\\
G_T(\frac{\pi}{2}) &: [e^{i\pi/3},\cos\frac{\pi}{3}+u_T\sin\frac{\pi}{3}], [e^{i\pi/3},\cos\frac{\pi}{3}+u_T'\sin\frac{\pi}{3}], [e^{i\pi/4}k,e^{i\pi/4}k]; &&144\\
G_T(\frac{\pi}{3}) &: [e^{i\pi/2},\cos\frac{\pi}{2}+i\sin\frac{\pi}{2}], [e^{i\pi/3},\cos\frac{\pi}{3}+u_T\sin\frac{\pi}{3}], [e^{i\pi/4}k,e^{i\pi/4}k]; &&96\\
G_O(\frac{\pi}{2}) &: [e^{i\pi/3},\cos\frac{\pi}{3}+u_T\sin\frac{\pi}{3}], [e^{i\pi/4},\cos\frac{\pi}{4}+i\sin\frac{\pi}{4}], [k,k]; &&576\\
G_O(\frac{\pi}{3}) &: [e^{i\pi/2},\cos\frac{\pi}{2}+u_O\sin\frac{\pi}{2}], [e^{i\pi/4},\cos\frac{\pi}{4}+i\sin\frac{\pi}{4}], [k,k]; &&384\\
G_O(\frac{\pi}{4}) &: [e^{i\pi/2},\cos\frac{\pi}{2}+u_O\sin\frac{\pi}{2}], [e^{i\pi/3},\cos\frac{\pi}{3}+u_T\sin\frac{\pi}{3}], [k,k]; &&288\\
G_I(\frac{\pi}{2}) &: [e^{i\pi/3},\cos\frac{\pi}{3}+u_I\sin\frac{\pi}{3}], [e^{i\pi/5},\cos\frac{\pi}{5}+u_I'\sin\frac{\pi}{5}], [k,k]; &&3600\\
G_I(\frac{\pi}{3}) &: [e^{i\pi/2},\cos\frac{\pi}{2}+i\sin\frac{\pi}{2}], [e^{i\pi/5},\cos\frac{\pi}{5}+u_I'\sin\frac{\pi}{5}], [k,k]; &&2400\\
G_I(\frac{\pi}{5}) &: [e^{i\pi/2},\cos\frac{\pi}{2}+i\sin\frac{\pi}{2}], [e^{i\pi/3},\cos\frac{\pi}{3}+u_I\sin\frac{\pi}{3}], [k,k]; &&1440
\end{align*}

By Lemma \ref{Lem:trans}, these groups preserve the enlarged Hopf fibration of $S^3_2$, and induce the corresponding reflection groups of $S^2$. For each group in the above list, the elements inducing the identity on $S^2$ correspond to movements along all fibres  $\{e^{i\theta/2}r \mid \theta\in \mathbb{R}\}$ in $S^3_2$. They form a cyclic group. The orders in the above list can be obtained by computing the orders of the corresponding cyclic groups.

\begin{lemma}\label{Lem:pirot}
Let {\bf R} be a reflection group of $S^2$. It has the corresponding graph $\Gamma_R$. Let $\widetilde{\Gamma}_R$ be a lifted graph, which is the union of great circles $\{c_\lambda\}_{\lambda\in\Lambda}$, where $\Lambda$ is some index set. Let $G_R$ be the corresponding group of $\widetilde{\Gamma}_R$, and let $\widetilde{{\bf R}}$ be the subgroup of $SO(4)$ generated by $\pi$-rotations around $c_\lambda$, $\lambda\in\Lambda$. Then $G_R=\widetilde{{\bf R}}$.
\end{lemma}

\begin{proof}
In $S^2$, if two points $u$ and $u'$ are in the same orbit under the action of {\bf R}, then rotations around $\mathcal{P}^{-1}(u)$ and $\mathcal{P}^{-1}(u')$ are conjugate by some element of $G_R$.

Let $r$ be a vertex of $\widetilde{\Gamma}_R$, and $u=\mathcal{P}(r)$. Suppose $\{c_1,\cdots,c_t\}$ are the great circles passing $r$. They are orthogonal to the fibre $\mathcal{P}^{-1}(u)$. If $G_R$ contains the $\pi$-rotation around $c_1$ and the $2\pi/t$-rotation around $\mathcal{P}^{-1}(u)$, then $G_R$ contains the $\pi$-rotations around $c_\lambda$, $\lambda=1,\cdots,t$.

By Lemma \ref{Lem:trans}, the last generator of $G_R$ is a $\pi$-rotation around some $c_\lambda$ passing the point $2 \in S^3_2$. Other generators are rotations around the fibres $\mathcal{P}^{-1}(u)$, where $u$ is one of $i$, $-i$, $k$, $e^{-i\pi/l}k$, $u_T$, $u_T'$, $u_O$, $u_I$, $u_I'$.

Hence by the constructions of the lifts, $G_R$ must contain all of the $\pi$-rotations around $c_\lambda$, $\lambda\in\Lambda$. Then $\widetilde{{\bf R}}\subseteq G_R$. On the other hand, the three generators of $G_R$ are compositions of elements in $\widetilde{{\bf R}}$. Hence $G_R\subseteq \widetilde{{\bf R}}$, and we have $G_R=\widetilde{{\bf R}}$.
\end{proof}

Note that by the constructions of the lifts, the lifted graphs must be preserved by their corresponding groups (or the $\pi$-rotations around the great circles in them).

\subsection{The fundamental quadrilaterals}\label{ssec:fund}

For the lifted graphs in $S^3_2$ we define their corresponding fundamental quadrilaterals as below. We choose certain closed piecewise geodesics in $\Gamma_C$, $\Gamma_D$, $\Gamma_T$, $\Gamma_O$, $\Gamma_I$, which pass the point $i$. Then we lift them from the point $2 \in S^3_2$. By Lemma \ref{Lem:lift}, the lifts will be simple closed piecewise geodesics in the lifted graphs, which give us the fundamental quadrilaterals.

For $\widetilde{\Gamma}_C(m,n)$ the closed piecewise geodesic starts at $i$ and passes $k$, $-i$, $e^{-i\pi/n}k$, $i$, $e^{-i\pi/m}e^{-i\pi/n}k$, $-i$, $e^{-i\pi/m}k$, $i$ successively; for $\widetilde{\Gamma}_D(l)$ the closed piecewise geodesic starts at $i$ and passes $k$, $e^{-i\pi s/l}k$, $i$, $-e^{i\pi(s-1)/l}k$, $-k$, $i$ successively, where $s\in [0,1]$; for $\widetilde{\Gamma}_D(l,\pi/2)$ the closed piecewise geodesic starts at $i$ and passes $k$, $e^{-2i\pi s/l}k$, $-i$, $e^{-i\pi/l}k$, $i$ successively, where $s\in [0,1]$. In Figure \ref{fig:familypath}, we marked the routes of the piecewise geodesics in $\Gamma_C$ and $\Gamma_D$.

\begin{figure}[h]
\centerline{\includegraphics{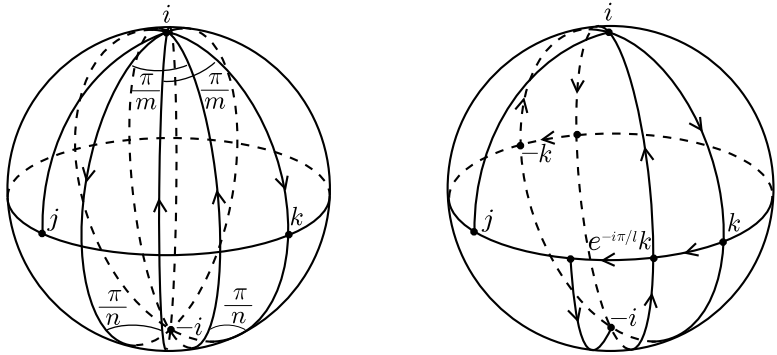}}
\caption{Piecewise geodesics in $\Gamma_C$ and $\Gamma_D$}\label{fig:familypath}
\end{figure}

For each of the other lifted graphs, to construct it we have chosen a vertex in the typical piece and defined the orbit of the vertex. Around the vertex there exists another piece having the opposite angles with the typical piece. Then the closed piecewise geodesic consists of the boundaries of the two pieces. By the construction of the lifted graphs, the route of the closed piecewise geodesic is determined, up to the choosing of the direction. In Figure \ref{fig:specialpath}, we marked the routes of the piecewise geodesics in $\Gamma_O$. At the self-intersections of the routes one should go straight.

\begin{figure}[h]
\centerline{\includegraphics{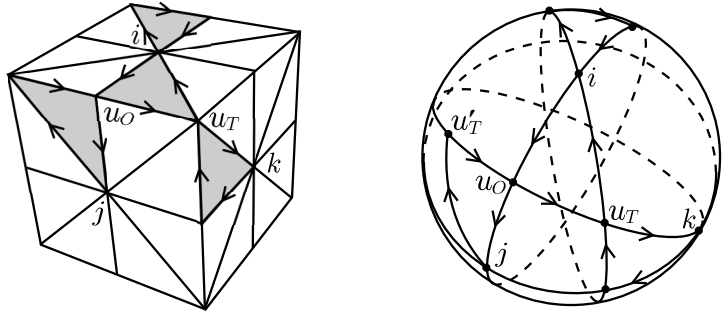}}
\caption{Piecewise geodesics in $\Gamma_O$}\label{fig:specialpath}
\end{figure}

By Lemma \ref{Lem:lift}, edge lengths and angles of the fundamental quadrilaterals can be computed from the closed piecewise geodesics. Let $KLMN$ denote a fundamental quadrilateral with edges $KL$, $LM$, $MN$, $NK$ and angles $\angle K$, $\angle L$, $\angle M$, $\angle N$. The lengths and angles of the fundamental quadrilaterals are given below.
\begin{align*}
&KLMN && KL && LM && MN && NK && \angle K && \angle L && \angle M && \angle N\\
&Q_C(m,n) && \pi && \pi && \pi && \pi && \frac{\pi}{m} && \frac{\pi}{n} && \frac{\pi}{m} && \frac{\pi}{n}\\
&Q_D(l) && \frac{\pi}{l} && \pi && \frac{\pi}{l} && \pi && \frac{\pi}{2} && \frac{\pi}{2} && \frac{\pi}{2} && \frac{\pi}{2}\\
&Q_D(l,\frac{\pi}{2}) && \frac{\pi}{2} && \frac{2\pi}{l} && \frac{\pi}{2} && \pi && \frac{\pi}{l} && \frac{\pi}{2} && \frac{\pi}{2} && \frac{\pi}{l}\\
&Q_T(\frac{\pi}{2}) && \psi(\frac{1}{3}) && \psi(-\frac{1}{3}) && \psi(\frac{1}{3}) && \psi(-\frac{1}{3}) && \frac{\pi}{3} && \frac{\pi}{3} && \frac{\pi}{3} && \frac{\pi}{3}\\
&Q_T(\frac{\pi}{3}) && \psi(\frac{1}{\sqrt{3}}) && \psi(\frac{-1}{\sqrt{3}}) && \psi(\frac{1}{\sqrt{3}}) && \psi(\frac{-1}{\sqrt{3}}) && \frac{\pi}{2} && \frac{\pi}{3} && \frac{\pi}{2} && \frac{\pi}{3}\\
&Q_O(\frac{\pi}{2}) && \frac{\pi}{2} && \psi(\frac{1}{\sqrt{3}}) && \psi(\frac{1}{3}) && \psi(\frac{1}{\sqrt{3}}) && \frac{\pi}{4} && \frac{\pi}{4} && \frac{\pi}{3} && \frac{\pi}{3}\\
&Q_O(\frac{\pi}{3}) && \frac{\pi}{4} && \frac{\pi}{2} && \frac{\pi}{4} && \frac{\pi}{2} && \frac{\pi}{4} && \frac{\pi}{2} && \frac{\pi}{4} && \frac{\pi}{2}\\
&Q_O(\frac{\pi}{4}) && \psi(\frac{2}{\sqrt{6}}) && \psi(-\frac{1}{3}) && \psi(\frac{2}{\sqrt{6}}) && \frac{\pi}{2} && \frac{\pi}{2} && \frac{\pi}{3} && \frac{\pi}{3} && \frac{\pi}{2}\\
&Q_I(\frac{\pi}{2}) && \psi(\mu\nu) && \psi(\frac{1}{\sqrt{5}}) && \psi(\mu\nu) && \psi(\frac{\sqrt{5}}{3}) && \frac{\pi}{3} && \frac{\pi}{5} && \frac{\pi}{5} && \frac{\pi}{3}\\
&Q_I(\frac{\pi}{3}) && \psi(\kappa) && \psi(\nu) && \psi(\kappa) && \psi(\nu) && \frac{\pi}{2} && \frac{\pi}{5} && \frac{\pi}{2} && \frac{\pi}{5}\\
&Q_I(\frac{\pi}{5}) && \psi(\mu) && \psi(\sigma) && \psi(\mu) && \psi(\sigma) && \frac{\pi}{2} && \frac{\pi}{3} && \frac{\pi}{2} && \frac{\pi}{3}
\end{align*}
Here $\psi(t)=\arccos(t) \in [0,\pi]$, and $\mu$, $\nu$, $\kappa$, $\sigma$ are the following constants.
$$\mu=\frac{\sqrt{5}+1}{2\sqrt{3}}, \nu=\frac{\sqrt{10+2\sqrt{5}}}{2\sqrt{5}}, \kappa=\sqrt{1-\nu^2}, \sigma=\sqrt{1-\mu^2}.$$

\begin{lemma}\label{Lem:4gon}
Let {\bf R}, $\Gamma_R$, $\widetilde{\Gamma}_R$, $G_R$ be as in Lemma \ref{Lem:pirot}. Let $|G_R|$ denote the order of $G_R$. Let $Q_R$ be the corresponding fundamental quadrilateral of $\widetilde{\Gamma}_R$. It has two angles of $\pi/m$ and two angles of $\pi/n$, where $m$ and $n$ are positive integers. Let $\mathcal{H}$ be the set of closed hemispheres in $S^3_2$, and let
$$\mathcal{C}(Q_R)=\bigcap\{H\in \mathcal{H}\mid Q_R\subset H\}$$
be the convex hull of $Q_R$. Then

(a) $\widetilde{\Gamma}_R$ is the union of the orbit of $Q_R$ under the action of $G_R$. If $mn\neq 1$, then the orbit of $Q_R$ contains $|G_R|$ elements. Otherwise, $\widetilde{\Gamma}_R=Q_R$.

(b) $G_R$ is generated by $\pi$-rotations around the edges of $Q_R$.

(c) If $m, n\geq 2$, then $\mathcal{C}(Q_R)$ is a geodesic tetrahedron in an open hemisphere of $S^3_2$; if $mn=1$, then $\mathcal{C}(Q_R)$ is a great circle; otherwise, $\mathcal{C}(Q_R)$ is a geodesic bi-gon.

(d) When $mn\neq 1$, elements in the orbit of $\mathcal{C}(Q_R)$ only meet at their boundaries, and there exists an embedded disk $D$ in $\mathcal{C}(Q_R)$ such that $D\cap\partial \mathcal{C}(Q_R)=\partial D=Q_R$. The union of the orbit of $D$ form an embedded closed surface in $S^3_2$ with genus
$$g=1+\frac{1}{2}(1-\frac{1}{m}-\frac{1}{n})|G_R|.$$
\end{lemma}

\begin{proof}
By Lemma \ref{Lem:lift} and the constructions of the lifted graphs, for each $u \in \Gamma_R$ the number of intersections in $\widetilde{\Gamma}_R\cap\mathcal{P}^{-1}(u)$ can be computed. By Lemma \ref{Lem:pirot} and the structure of $G_R$, the stable subgroup of $\mathcal{P}^{-1}(u)$ acts on $\mathcal{P}^{-1}(u)$ as a dihedral group action, which is generated by reflections about the intersections.

For a quadrilateral $KLMN$ in the orbit of $Q_R$, suppose that a $\pi$-rotation around $KL$ maps it to $KLM'N'$, then in $S^2$ the reflection about $\mathcal{P}(KL)$ maps $\mathcal{P}(KLMN)$ to $\mathcal{P}(KLM'N')$, and $KLM'N'$ can be obtained by lifting $\mathcal{P}(KLM'N')$ from $K$. Hence the orbit of $Q_R$ can be obtained by lifting the orbit of $\mathcal{P}(Q_R)$. Then we can see how the elements in the orbit of $Q_R$ fit together.

(a) By Lemma \ref{Lem:pirot}, $G_R$ is generated by the $\pi$-rotations around great circles in $\widetilde{\Gamma}_R$. Then via $\pi$-rotations around edges of quadrilaterals, every great circle lies in the union of the orbit of $Q_R$, and so does $\widetilde{\Gamma}_R$. Then by the description as above, the number of elements in the orbit of $Q_R$ can be computed by Lemma \ref{Lem:lift}, and we can see that when $mn\neq 1$, the stable subgroup of $Q_R$ in $G_R$ is trivial.

(b) Since the group generated by $\pi$-rotations around the edges of $Q_R$ contains the three generators of $G_R$. By Lemma \ref{Lem:pirot}, it is $G_R$.

(c) When $m,n\geq 2$, $Q_R$ is a quadrilateral in $S^3_2$ with edge lengths not bigger than $\pi$ and angles not bigger than $\pi/2$. Then the distance between any two vertices of $Q_R$ is not bigger than $\pi$. Hence $\mathcal{C}(Q_R)$ is a geodesic tetrahedron in an open hemisphere of $S^3_2$. The other two cases are trivial.

(d) By the description at the beginning of the proof, elements in the orbit of $\mathcal{C}(Q_R)$ only meet at their boundaries. The existence of $D$ is trivial. Since there are $|G_R|$ elements in the orbit of $Q_R$, the surface contains $2|G_R|/2m$ vertices with angle $\pi/m$, $2|G_R|/2n$ vertices with angle $\pi/n$, $4|G_R|/2$ edges and $|G_R|$ faces. By the Euler formula, we have
$$2-2g=\frac{1}{m}|G_R|+\frac{1}{n}|G_R|-2|G_R|+|G_R|.$$
Then we can get $g$, and the proof is finished.
\end{proof}

In Lemma \ref{Lem:4gon}, if $Q_R$ satisfies certain conditions, then we can choose $D$ to be minimal such that the union of its orbit is a closed minimal surface in $S^3_2$. We will discuss such conditions in the next section.

\section{The existence of minimal surfaces}\label{Sec:existence}

Let $KLMN$ be a fundamental quadrilateral corresponding to some lifted graph. Then it has two angles of $\pi/m$ and two angles of $\pi/n$, where $m$ and $n$ are positive integers. If $m=1$ or $n=1$, then the lifted graph is one of $Q_C(m,1)$, $Q_C(1,n)$ and $Q_D(1,\pi/2)$. By Lemma \ref{Lem:4gon} we can get a geodesic two dimensional sphere, which is minimal in $S^3_2$. At what follows, we assume that $m\geq 2$ and $n\geq 2$. Then by Lawson's existence theorem of minimal surfaces we will show that in $S^3_2$ there exists an embedded closed minimal surface corresponding to $KLMN$.

\subsection{Lawson's existence theorem}

Here we give a brief introduction to Lawson's results. For more details one should see the original paper \cite{L1} and \cite{L2}.

Let $\Gamma$ be a geodesic polygon in $S^3_2$ having vertices $v_0, v_1, \cdots, v_s=v_0$ and edges $\gamma_0, \gamma_1, \cdots, \gamma_s=\gamma_0$ such that for each $t$, $\gamma_t$ meets $\gamma_{t-1}$ in $v_t$ at an angle of the form $\pi/(l_t+1)$ where $l_t$ is a positive integer. If $\gamma$ and $\delta$ are distinct geodesics which meet in $S^3_2$, then let $S(\gamma,\delta)$ denote the unique geodesic two sphere containing $\gamma\cup\delta$, and for each $t$, let $N_t$ be the geodesic perpendicular to $S(\gamma_{t-1},\gamma_t)$ at $v_t$.

A subset of $S^3_2$ is bounded by $S(\gamma,\delta)$ if it is contained in one of the two closed hemispheres determined by $S(\gamma,\delta)$. The polygon $\Gamma$ is called proper if for each $t$, it is bounded either by $S(\gamma_{t-1},N_t)$ or by $S(\gamma_t,N_t)$. It is called convex if $\Gamma\subset\partial \mathcal{C}(\Gamma)$, where $\mathcal{C}(\Gamma)=\bigcap\{H\in \mathcal{H}\mid\Gamma \subset H\}$ and $\mathcal{H}$ is the set of closed hemispheres in $S^3_2$.

Finally, let $\mathcal{S}_{\Gamma}$ be the set of geodesic two spheres in $S^3_2$, such that $S \in \mathcal{S}_{\Gamma}$ if and only if $S\cap \Gamma$ has at least four components. Denote by $\mathcal{C}(\Gamma)^\circ$ the interior of $\mathcal{C}(\Gamma)$, and denote by $\triangle$ the closed unit disk in $\mathbb{C}$.

\begin{theorem}\label{Thm:exist}
For a proper convex geodesic polygon $\Gamma$ in $S^3_2$ having vertex angles of the type $\pi/(l+1)$, where $l$ is a positive integer depending on the vertex, let $G_{\Gamma}$ be the group generated by $\pi$-rotations around edges of $\Gamma$. If $\Gamma$ satisfies conditions (A--D) as below, then there exists an embedded minimal disk $\mathcal{M}_\Gamma$ in $\mathcal{C}(\Gamma)$ such that $\mathcal{M}_\Gamma\cap\partial \mathcal{C}(\Gamma)= \partial \mathcal{M}_\Gamma =\Gamma$. Moreover, let $M_{\Gamma}$ be the union of the orbit of $\mathcal{M}_\Gamma$ under the action of $G_\Gamma$, then $M_{\Gamma}$ is a complete non-singular minimal surface in $S^3_2$, and it is compact if and only if $G_\Gamma$ is finite.

(A) $\Gamma$ lies in an open hemisphere of $S^3_2$.

(B) For each $p\in \mathcal{C}(\Gamma)^\circ$ there is a geodesic two sphere $S_p \notin \mathcal{S}_{\Gamma}$ containing $p$.

(C) For each $t$, if one of $S(\gamma_{t-1},N_t)$ and $S(\gamma_t,N_t)$ fails to bound $\Gamma$, then $l_t=1$.

(D) There exists a continuous map $\rho: \mathcal{C}(\Gamma)\rightarrow \triangle$ which is differentiable in $\mathcal{C}(\Gamma)^\circ$ and carries $\Gamma$ monotonically onto $\partial\triangle$ such that for each $S\in \mathcal{S}_\Gamma$ the differential of the map $\rho\mid S\cap\mathcal{C}(\Gamma)^\circ$ is everywhere of rank $2$.
\end{theorem}

\begin{remark}
The above theorem is corresponding to Theorem 1 in \cite{L1}. For our use, we have changed the statement a little bit. For examples, we state the results for $S^3_2$ rather than $S^3$, and the ``$\pi$-rotation'' is called ``reflection'' in \cite{L1}. However, all the contents are essentially contained in section 4 of \cite{L1}.
\end{remark}

By Lemma \ref{Lem:4gon}, we also need to show that the fundamental quadrilateral $KLMN$ is proper and satisfies conditions (B--D). However, $KLMN$ may not satisfy (C) in some cases. The condition (C) is used in Lemma 4.3 of \cite{L1} to show that:

(C*) At each vertex $v$ of $M_\Gamma$ there exists a small disk neighborhood $B$ of $v$ and a geodesic two sphere $S$ passing $v$ such that $\partial B\cap S$ contains only two points.

Then by Theorem 3 of \cite{L2}, $M_\Gamma$ is non-singular. In our cases, we will verify the condition (C*), and then obtain the conclusions of Theorem \ref{Thm:exist}.

\subsection{The maximum principle and basic lemmas}

To verify the conditions, we need some lemmas about the maximum principle, which states that if two minimal surfaces $M_1$ and $M_2$ meet at an interior point of each surface and $M_1$ locally lies on one side of $M_2$ near the intersection, then the surfaces $M_1$ and $M_2$ coincide near the intersection, see \cite{KPS} and \cite{M}.

Let $B_+$ and $B_-$ be the open hemispheres in $S^3_2$ centred at $2$ and $-2$ respectively, and let $\overline{B}_+$ be the closure of $B_+$. Let $D_+$ be the two dimensional open hemisphere centered at $2$ which is orthogonal to the fibre $\mathcal{P}^{-1}(i)$, and let
$$S^1=\{[e^{-i\tau}, e^{i\tau}]\in SO(4)\mid \tau\in \mathbb{R}\}.$$
Then modular the $S^1$-action, we have a map $\rho : S^3_2\setminus\mathcal{P}^{-1}(-i)\rightarrow D_+$, which gives a fibration. Figure \ref{fig:fibration} gives a sketch map of this fibration. Then we have the following version of the maximum principle.

\begin{figure}[h]
\centerline{\includegraphics{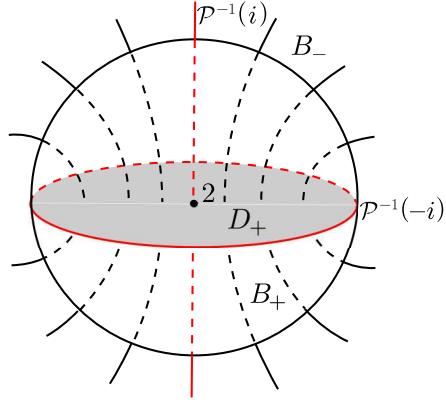}}
\caption{The fibration given by $\rho$}\label{fig:fibration}
\end{figure}

\begin{lemma}\label{Lem:maxP}
Suppose that $M_1$ and $M_2$ are two minimal surfaces in $B_+$ and $B_-$ respectively, and the boundary of each surface is a simple closed curve. Then $M_2$ can be rotated via the $S^1$-action, and when the first time $M_2$ intersects $M_1$, they must intersect at some boundary point of $M_1$ or $M_2$.
\end{lemma}

\begin{lemma}\label{Lem:unique}
If $\Gamma$ is a simple closed curve which bounds two minimal surfaces $M_1$ and $M_2$ in $B_+$, and the map $\rho\mid M_1\cup M_2$ is one to one on $(\rho\mid M_1\cup M_2)^{-1}\rho(\Gamma)$, then $M_1$ and $M_2$ coincide.
\end{lemma}

\begin{proof}
By the $\pi$-rotation around $\mathcal{P}^{-1}(-i)$, we can map $M_2$ to a minimal surface $M_2'$ in $B_-$. Then via the $S^1$-action, there are two directions to move $M_2'$ back to $M_2$. By Lemma \ref{Lem:maxP}, when it first meets $M_1$, the boundaries of the two surfaces must coincide. Hence $M_1$ and $M_2$ coincide.
\end{proof}

To our propose, we will mainly consider the case that $\Gamma$ is a geodesic quadrilateral in $\overline{B}_+$ with vertices $v_0$, $v_1$, $v_2$, $v_3$ and edges $\gamma_0$, $\gamma_1$, $\gamma_2$, $\gamma_3$. Let $s_t$ denote the length of $\gamma_t$, and let $\tau_t$ denote the angle at $v_t$, where $t=0,1,2,3$.

\begin{lemma}\label{Lem:position}
For a geodesic quadrilateral $\Gamma$ in $\overline{B}_+$, if $s_t\leq \pi$, $\tau_t\leq\pi/2$, $0\leq t\leq 3$, and in $\mathcal{C}(\Gamma)$ the two geodesics $v_0v_2$ and $v_1v_3$ have a common perpendicular $o_1o_2$ such that $o_t$ lies in the interior of $v_{t-1}v_{t+1}$, $t=1,2$, then there exists $\eta'\in SO(4)$ such that $\eta'(\Gamma)\subset \overline{B}_+\setminus\mathcal{P}^{-1}(-i)$ and $\eta'(o_1o_2)=\{2e^{is/2}\in S^3_2\mid s\in [0,d]\}$, where $d$ is the length of $o_1o_2$. Moreover, for $\Gamma'=\eta'(\Gamma)$ the map $\rho\mid \mathcal{C}(\Gamma')$ is one to one on $(\rho\mid \mathcal{C}(\Gamma'))^{-1}\rho(\Gamma')$, and in $\mathcal{C}(\Gamma')$ the minimal surface bounded by $\Gamma'$ is unique.
\end{lemma}

\begin{figure}[h]
\centerline{\includegraphics{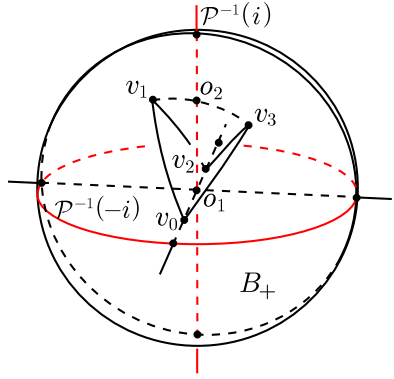}}
\caption{The geodesic quadrilateral $\eta'(\Gamma)$ in $\overline{B}_+$}\label{fig:eta1gam}
\end{figure}

\begin{proof}
We can choose an $\eta'\in SO(4)$ such that $\eta'(o_1o_2)=\{2e^{is/2}\in S^3_2\mid s\in [0,d]\}$ and $\eta'(o_1)=2$. Since $s_t\leq \pi$, the distances between $o_1$ and $v_1$, $v_3$ and the distances between $o_2$ and $v_0$, $v_2$ are not bigger than $\pi$. Hence $d\leq \pi$, and $\eta'(\Gamma)\subset \overline{B}_+$. Since $\tau_t\leq\pi/2$, lengths of $v_0v_2$ and $v_1v_3$ are at most $\pi$. Hence $\eta'(\Gamma)\cap \mathcal{P}^{-1}(-i)=\emptyset$.

If a fibre $f$ of the fibration given by $\rho$ intersects $\Gamma'$, then $f$ does not pass $2$. Consider the geodesic two sphere $S$ which contains $f$ and $2$. Since $f$ is orthogonal to the boundary of $B^+$ which contains $\mathcal{P}^{-1}(-i)$, in $S$ the fibre $f$ can only intersect $\mathcal{C}(\Gamma')\cap S$ at one point. Hence the map $\rho\mid \mathcal{C}(\Gamma')$ is one to one on $(\rho\mid \mathcal{C}(\Gamma'))^{-1}\rho(\Gamma')$. Then $\Gamma'$ can be rotated into $B_+$ via the $S^1$-action, and by Lemma \ref{Lem:unique}, in $\mathcal{C}(\Gamma')$ the minimal surface bounded by $\Gamma'$ is unique.
\end{proof}


\subsection{Verification of the conditions}

At what follows, we show that $KLMN$ is proper and satisfies the conditions (B), (D) and (C*).

\begin{lemma}\label{Lem:Sposition}
For a given vertex $v$ of the fundamental quadrilateral $KLMN$ there exists $\eta\in SO(4)$ such that $\eta$ preserves the enlarged Hopf fibration of $S^3_2$, maps $v$ to $v_0=2$, and maps the two edges containing $v$ to the two geodesics:
\begin{align*}
&2(\cos(s/2)+j\sin(s/2)), s \in [0,s_0], \\ &2(\cos(s/2)+e^{i\pi/l}j\sin(s/2)), s \in [0,s_3],
\end{align*}
where $\pi/l \in [0,\pi/2]$ is the angle at $v$, $s_0$ and $s_3$ are the lengths of the two edges, and $s_0\geq s_3$. Moreover, $\Gamma=\eta(KLMN)$ satisfies the conditions in Lemma \ref{Lem:position}. In $\mathcal{C}(\Gamma)$ the minimal surface bounded by $\Gamma$ is unique, and so does $KLMN$.
\end{lemma}

\begin{proof}
The existence of $\eta$ can be obtained by Lemma \ref{Lem:trans}, and we only need to show that in $\mathcal{C}(\Gamma)$ the two geodesics $v_0v_2$ and $v_1v_3$ have a common perpendicular $o_1o_2$ such that $o_t$ lies in the interior of $v_{t-1}v_{t+1}$, $t=1,2$.

For fundamental quadrilaterals other than $Q_D(l,\pi/2)$, $Q_O(\pi/2)$, $Q_O(l,\pi/4)$ and $Q_I(\pi/2)$, let $o_t$ be the middle point of $v_{t-1}v_{t+1}$, $t=1,2$, then the geodesic $o_1o_2$ in $\mathcal{C}(\Gamma)$ is a required common perpendicular.

For $Q_D(l,\pi/2)$, $Q_O(\pi/2)$, $Q_O(\pi/4)$ and $Q_I(\pi/2)$, let $u$ be the self-intersection of the closed piecewise geodesic $\mathcal{P}(KLMN)$ in $S^2$. Then the fibre $\mathcal{P}^{-1}(u)$ intersects $KLMN$ orthogonally. It divides $KLMN$ into two quadrilaterals which can be mapped to each other by the $\pi$-rotation around $\mathcal{P}^{-1}(u)$. Notice that there exists a common perpendicular of $KM$ and $\mathcal{P}^{-1}(u)$ in $\mathcal{C}(KLMN)$, hence the union of its orbit under the $\pi$-rotation gives a common perpendicular of $KM$ and $LN$.
\end{proof}

\begin{proposition}
The fundamental quadrilateral $KLMN$ is proper.
\end{proposition}

\begin{proof}
Consider $\Gamma=\eta(KLMN)$ as in Lemma \ref{Lem:Sposition}. Notice that the fibre passing $2(\cos(s_0/2)+j\sin(s_0/2))$ has the tangent vector $i(\cos(s_0/2)+j\sin(s_0/2))$. This means that moving along $\gamma_0$ from $v_0$ to $v_1$, the tangent subspace which is orthogonal to the fibre rotates around $\gamma_0$ with an angle $s_0/2$. For $\gamma_3$, the situation is similar. Since $s_t\leq\pi$ for $0\leq t\leq 3$, all the rotation angles are not bigger than $\pi/2$. Then the quadrilateral must be proper.
\end{proof}

\begin{proposition}
The fundamental quadrilateral $KLMN$ satisfies (B) and (D).
\end{proposition}

\begin{proof}
For (B), consider the family of geodesic two spheres containing $KM$.

For (D), consider the quadrilateral $\Gamma'=\eta'(\eta(KLMN))$ given by Lemma \ref{Lem:Sposition} and Lemma \ref{Lem:position}. Then $\rho$ maps $\mathcal{C}(\Gamma')$ into $D_+$. It is differentiable in $\mathcal{C}(\Gamma')^\circ$ and carries $\Gamma'$ monotonically onto $\rho(\Gamma')=\partial\rho(\mathcal{C}(\Gamma'))$. For each $S\in \mathcal{S}_\Gamma$, consider the intersection $(S\cap\mathcal{C}(\Gamma'))\cap f$, where $f$ is a fibre of the fibration given by $\rho$. If $f$ passes $2$, then the intersection is one point. Otherwise, by a similar discussion as in Lemma \ref{Lem:position}, the intersection contains at most one point. Hence the differential of the map $\rho\mid S\cap\mathcal{C}(\Gamma')^\circ$ is everywhere of rank $2$. Then by a suitable continuous map from $D_+$ to $\mathbb{C}$, we can get the required map.
\end{proof}

Given a vertex $v$ of $KLMN$, let $\Gamma=\eta(KLMN)$ as in Lemma \ref{Lem:Sposition}. Then by section 4 of \cite{L1}, the conclusions of Theorem \ref{Thm:exist} hold, except that points in the orbit of vertices of $\Gamma$ may be singular points of $M_\Gamma$. We need the following proposition to guarantee that the minimal surface is non-singular.

\begin{proposition}
The fundamental quadrilateral $KLMN$ satisfies (C*).
\end{proposition}

\begin{proof}
For a given vertex $v$ of $KLMN$, if the condition (C) holds at $v$, then by Lemma 4.3 of \cite{L1}, the condition (C*) holds at $v$. Hence we only need to consider fundamental quadrilaterals other than $Q_C(m,n)$, $Q_D(l)$, $Q_D(l,\pi/2)$, and vertices with angles smaller than $\pi/2$.

\begin{figure}[h]
\centerline{\includegraphics{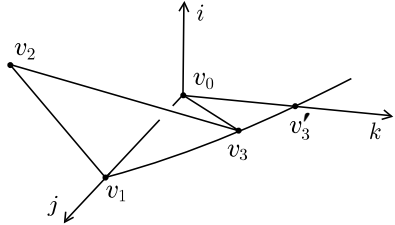}}
\caption{The geodesic quadrilateral $\Gamma$}\label{fig:compare}
\end{figure}

Let $\Gamma=\eta(KLMN)$ as in Lemma \ref{Lem:Sposition}. Then $s_0$ and $s_2$ are longer than $s_1$ and $s_3$, and the points $v_2$ and $2e^{i\pi/4}$ lie in the same component of $B_+\setminus D_+$. Since $0<s_3<s_0<\pi$, $S(\gamma_3,N_0)$ bounds $\Gamma$ and $S(\gamma_0,N_0)$ fails to bound $\Gamma$.

The great circle passing $v_1$ and $v_3$ intersects $2(\cos(s/2)+k\sin(s/2)), s \in [0,\pi]$ at an interior point $v_3'$. Consider the geodesic quadrilateral $v_0v_1v_2v_3'$ in $B_+$, which is denoted by $\Upsilon$. One can check that $\Upsilon$ is also a proper convex quadrilateral. Let $\Gamma'=\eta'(\Gamma)$ as in Lemma \ref{Lem:position}, and let $\Upsilon'=\eta'(\Upsilon)$. Then since the geodesic $v_1v_3'$ in $B_+$ has length less than $\pi$, the ``moreover'' part of Lemma \ref{Lem:position} also holds for $\Upsilon'$. Hence $\Upsilon$ also satisfies conditions (A), (B) and (D).

By the discussion in section 4 of \cite{L1}, $\Upsilon$ bounds a minimal disk $D_{\Upsilon}$ in $\mathcal{C}(\Upsilon)$. Notice that $\Upsilon$ satisfies condition (C) at $v_0$. By Lemma 4.3 of \cite{L1}, if we extend $D_{\Upsilon}$ by successive $\pi$-rotations around edges passing $v_0$, then we can get a minimal disk containing $v_0$ as an interior point, and at $v_0$ the minimal disk is non-singular. Hence $D_{\Upsilon}$ has a tangent space at $v_0$.

Denote by $D_{\Gamma}$ the minimal disk bounded by $\Gamma$ in $\mathcal{C}(\Gamma)$. Now consider the minimal disks $D_{\Gamma}'=\eta'(D_{\Gamma})$ and $D_{\Upsilon}'=\eta'(D_{\Upsilon})$. Via the $S^1$-action we can rotate $D_{\Upsilon}'$ into $B_-$ such that $\Upsilon'$ does not touch $\mathcal{C}(\Gamma')^\circ$ and $\Gamma'$ does not touch $\mathcal{C}(\Upsilon')^\circ$. Then we rotate $D_{\Upsilon}'$ back. By Lemma \ref{Lem:maxP}, the interior of $D_{\Gamma}'$ and $D_{\Upsilon}'$ does not intersect. Then $D_{\Gamma}$ must have a tangent space at $v_0$. Hence $S(\gamma_0,N_0)$ is a required geodesic two sphere, and the condition (C*) holds at $v_0$.
\end{proof}

By Theorem \ref{Thm:exist} and Lemma \ref{Lem:4gon}, for each fundamental quadrilateral we get a corresponding embedded closed non-singular minimal surface in $S^3_2$, and the genus of the surface can be computed. We list the results as below.
\begin{align*}
&M_C(m,n) && (m-1)(n-1) && M_T(\frac{\pi}{2}) && 25 && M_O(\frac{\pi}{2}) && 121 && M_I(\frac{\pi}{2}) && 841 \\
&M_D(l) && 1 && M_T(\frac{\pi}{3}) && 9 && M_O(\frac{\pi}{3}) && 49 && M_I(\frac{\pi}{3}) && 361 \\
&M_D(l,\frac{\pi}{2}) && (l-1)^2 && && && M_O(\frac{\pi}{4}) && 25 &&M_I(\frac{\pi}{5}) && 121
\end{align*}

In the above list, the surface $M_C(m,n)$ was constructed in \cite{L1}. When $m=2$ and $n=2$, it gives the Clifford torus. The surface $M_D(l)$ always gives the Clifford torus. The surface $M_D(l,\pi/2)$ ($M_O(\pi/4)$) is congruent to $M_C(l,l)$ ($M_T(\pi/2)$), and a fundamental piece of $M_C(l,l)$ ($M_T(\pi/2)$) is congruent to the union of two pieces of $M_D(l,\pi/2)$ ($M_O(\pi/4)$). Since the isometric groups of the surfaces $M_O(\pi/2)$ and $M_I(\pi/5)$ have different orders by Proposition \ref{Pro:iso}, we have seven new embedded closed minimal surfaces in total.

Clearly in Theorem \ref{Thm:exist} if $G_\Gamma$ has a subgroup $H_\Gamma$ which acts on $S^3$ freely, then $M_\Gamma/H_\Gamma$ is a minimal surface in the spherical manifold $S^3/H_\Gamma$. As an example, let $G_P$ be the group generated by the elements
$$[1,\cos\frac{\pi}{2}+i\sin\frac{\pi}{2}], [1,\cos\frac{\pi}{3}+u_I\sin\frac{\pi}{3}].$$
Then we have $G_P\subset G_I(\pi/5)$, $G_P\subset G_I(\pi/3)$, $G_P\subset G_I(\pi/2)$. Note that the order of $G_P$ is $120$. Hence we can get embedded closed minimal surfaces with genera $2$, $4$, $8$ in $S^3_2/G_P$, which is the Poincar\'e's homology three sphere.

\section{The symmetry of minimal surfaces}\label{Sec:symmetry}

\subsection{Isometric groups of the minimal surfaces}

Except the group actions given in section \ref{sssec:LofG}, there exist further symmetries on the minimal surfaces.

Let $KLMN$ be a fundamental quadrilateral given in section \ref{ssec:fund}.

When $KLMN$ is one of $Q_D(l,\pi/2)$, $Q_O(\pi/2)$, $Q_O(\pi/4)$ and $Q_I(\pi/2)$, let $u$ be the self-intersection of the closed piecewise geodesic $\mathcal{P}(KLMN)$. Then the fibre $\mathcal{P}^{-1}(u)$ splits $KLMN$ into two smaller quadrilaterals. The $\pi$-rotation around the fibre $\mathcal{P}^{-1}(u)$ preserves $KLMN$ and interchanges the two smaller quadrilaterals.

When $KLMN$ is one of $Q_C(m,n)$, $Q_T(\pi/3)$, $Q_O(\pi/3)$, $Q_I(\pi/3)$ and $Q_I(\pi/5)$, let $o_1$ and $o_2$ be the middle points of $KM$ and $LN$. Then the $\pi$-rotation around the great circle containing $o_1o_2$ preserves $KLMN$.

When $KLMN$ is $Q_C(m,n)$, there exists another order two symmetry, which is a reflection of $S^3_2$. If $m=n=l$, then $Q_C(m,n)$ can be divided into two smaller quadrilaterals which are congruent to $Q_D(l,\pi/2)$. Then the $\pi$-rotation preserving the smaller piece also preserves $Q_C(m,n)$.

By Lemma \ref{Lem:Sposition}, these symmetries preserve the minimal disks bounded by the quadrilaterals, hence they also preserve the closed minimal surfaces.

\begin{proposition}\label{Pro:iso}
When $(m-1)(n-1)\geq 2$, $m \neq n$ and $l \geq 3$, the following list gives the orders of the isometric groups of the minimal surfaces.
\begin{align*}
&M_C(m,n) && 8mn && M_O(\frac{\pi}{2}) && 1152 && M_I(\frac{\pi}{2}) && 7200 \\
&M_C(l,l) && 16l^2 && M_O(\frac{\pi}{3}) && 768 && M_I(\frac{\pi}{3}) && 4800 \\
&M_T(\frac{\pi}{3}) && 192 && M_O(\frac{\pi}{4}) && 576 && M_I(\frac{\pi}{5}) && 2880
\end{align*}
Moreover, the isometric groups are generated by the groups given in section \ref{sssec:LofG} and the isometric groups of the fundamental quadrilaterals given in section \ref{ssec:fund}.
\end{proposition}

\begin{proof}
By the discussion before the proposition, the list gives the lower bounds. Let $M$ be a minimal surface in the above list. Denote by $V$ be the set of vertices of quadrilaterals in $M$, which have angles smaller than $\pi/2$. Then by the discussion at the end of section 1 of \cite{L1}, $v \in V$ if and only if the Gauss curvature at $v$ is $1$. Hence any isometry of $M$ must preserve the set $V$. Then it must map one quadrilateral in $M$ to another quadrilateral in $M$. Then by Lemma \ref{Lem:4gon} and the discussion before the proposition, we know that the list also gives the upper bounds, and the isometric groups of the minimal surfaces are described as in the proposition.
\end{proof}

Let $M$ be a minimal surface listed in Proposition \ref{Pro:iso}, and denote by $Isom(M)$ the isometric group of $M$. By Proposition \ref{Pro:iso}, we know that the $Isom(M)$-action on $M$ is extendable over $S^3_2$. Denote by $Isom^+(M)$ the subgroup of $Isom(M)$ which preserves both orientations of $M$ and $S^3_2$. Then $Isom^+(M)$ has index $2$ or $4$ in $Isom(M)$, depending on whether there exist elements in $Isom(M)$ reverse both orientations of $M$ and $S^3_2$.

Compared with Table \ref{tab:maximal action} and Table \ref{tab:gen max act} in section \ref{Sec:intro}, when $M$ is one of $M_C(2,l)$, $M_C(l,2)$ and $M_C(l,l)$, the $Isom^+(M)$-action and $Isom(M)$-action correspond to the actions in the last two rows. Now with the new minimal surfaces we have:

\begin{proposition}\label{Pro:max}
(1) When $M$ is one of $M_T(\pi/3)$, $M_O(\pi/3)$, $M_O(\pi/4)$, $M_I(\pi/2)$, $M_I(\pi/3)$ and $M_I(\pi/5)$, the $Isom^+(M)$-action on $(M,S^3_2)$ is maximal.

(2) When $M$ is one of $M_O(\pi/4)$ and $M_I(\pi/5)$, the $Isom(M)$-action on $(M,S^3_2)$ is maximal in the general meaning.
\end{proposition}

\subsection{The conjugation of $Isom^+(M)$ and $G_M$}

For a minimal surface $M$ in Proposition \ref{Pro:iso} other than $M_C(l,l)$, let $G_M$ be the corresponding group of $M$ defined in section \ref{sssec:LofG}. Note that $M_C(l,l)$ and $M_D(l,\pi/2)$ are the same surface, for this minimal surface $M$, let $G_M$ be the corresponding group of $M_D(l,\pi/2)$. Clearly both $Isom^+(M)$ and $G_M$ are subgroups of $Isom(M)$.

Let $\widetilde{\Gamma}_M$ be the skeleton of $M$ defined in section \ref{ssec:skel}, and let $u$ be a point in $\mathcal{P}(\widetilde{\Gamma}_M)$ other than vertices. Suppose that $\widetilde{\Gamma}_M$ intersects the fibre $\mathcal{P}^{-1}(u)$ at $t$ points. Then by Lemma \ref{Lem:lift} and Lemma \ref{Lem:trans}, $[e^{2\pi i/t},1]$ keeps $\widetilde{\Gamma}_M$ invariant. Let $\widetilde{\Gamma}_M'$ be the image of $\widetilde{\Gamma}_M$ under the map $[e^{\pi i/t},1]$. Then the group generated by $\pi$-rotations around great circles in $\widetilde{\Gamma}_M'$ also preserves $\widetilde{\Gamma}_M$. Hence by Lemma \ref{Lem:Sposition} it preserves $M$, and it also preserves the orientation of $M$. Then by Proposition \ref{Pro:iso} the group must be $Isom^+(M)$, and by Lemma \ref{Lem:pirot} we have the following proposition.

\begin{proposition}\label{Pro:conjugation}
$Isom^+(M)$ and $G_M$ are conjugate to each other by $[e^{\pi i/t},1]$.
\end{proposition}

Denote by $G_M^+$ the intersection $Isom^+(M)\cap G_M$. Note that the last generator of $G_M$ reverses the orientation of $M$. Hence $G_M^+$ is generated by the first two generators of $G_M$. Moreover, it has index $2$ in both $Isom^+(M)$ and $G_M$.

At what follows, we will consider the relations between the three orbifold pairs
$$(M/Isom^+(M),S^3_2/Isom^+(M)), (M/G_M,S^3_2/G_M), (M/G_M^+,S^3_2/G_M^+),$$ and give some discussions about results in \cite{WWZZ} and \cite{WWZ}. For contents on orbifolds one can see \cite{BMP} and \cite{T}.

By the results in \cite{WWZZ}, $S^3_2/Isom^+(M)$ always has the underlying space $S^3$, and the singular set of it is a trivalent graph as in Figure \ref{fig:IsomM}. The numbers denote the indices of the singular edges, and singular edges with index $2$ are not labeled. The order of these orbifolds is the same as their corresponding minimal surfaces in Proposition \ref{Pro:iso}. In each case $M/Isom^+(M)$ is the common boundary of regular neighborhoods of the singular edges $a$ and $a'$.

\begin{figure}[h]
\centerline{\includegraphics{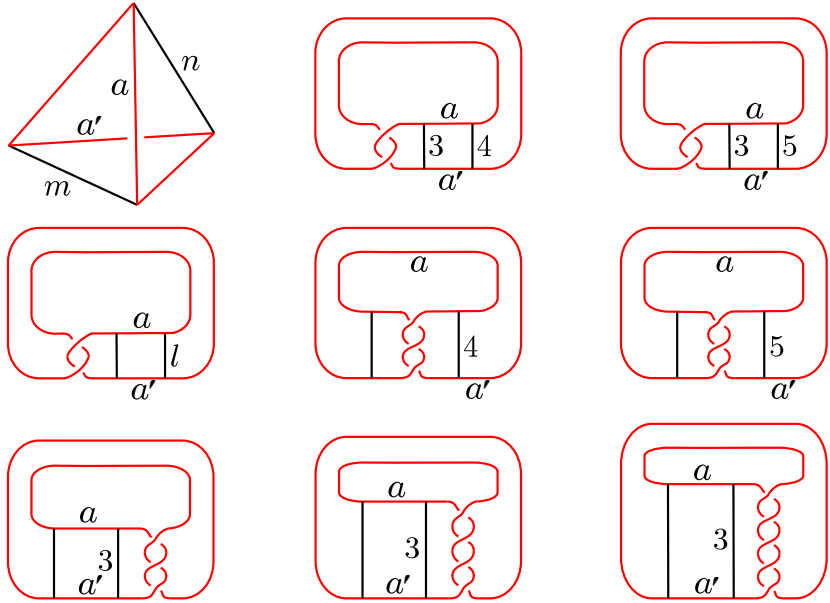}}
\caption{The three dimensional orbifolds $S^3_2/Isom^+(M)$}\label{fig:IsomM}
\end{figure}

In each orbifold in Figure \ref{fig:IsomM}, there is an unknotted (red) circle with index $2$ (except four points). The circle bounds a disk which does not contain singular points in its interior. Hence for each $S^3_2/Isom^+(M)$ we have a two sheet branched covering orbifold as in Figure \ref{fig:GMadd}. The orbifold also has underlying space $S^3$, and the singular set is a two components link. The dashed lines and circles in the figure are not contained in the singular sets. They are rotation axes, and the $\pi$-rotations of $S^3$ around them preserve the orbifolds. The red dashed lines correspond to the red circles in Figure \ref{fig:IsomM}, and we also marked the corresponding $a$ and $a'$.

\begin{figure}[h]
\centerline{\includegraphics{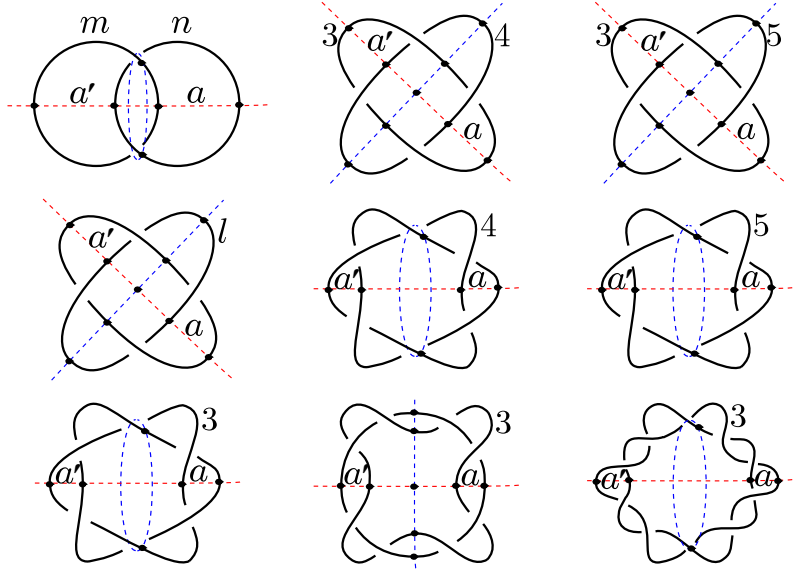}}
\caption{The three dimensional orbifolds $S^3_2/G_M^+$}\label{fig:GMadd}
\end{figure}

The orbifold in Figure \ref{fig:GMadd} is exactly $S^3_2/G_M^+$. This can be shown by the results in \cite{MS}, which give the correspondence between fibred spherical 3-orbifolds and quaternion representations of their fundamental groups. Actually by Lemma \ref{Lem:trans} the base 2-orbifold, the Euler number and the invariants of singular fibres of the fibred orbifold $S^3_2/G_M^+$ can be computed from the two generators of $G_M^+$. Similarly one can show that $S^3_2/G_M$ is isomorphic to $S^3_2/Isom^+(M)$ by the results in \cite{MS} or Lemma \ref{Lem:trans}, which is also a corollary of Proposition \ref{Pro:conjugation},

By Lemma \ref{Lem:pirot} and Lemma \ref{Lem:4gon}, $\widetilde{\Gamma}_M'/Isom^+(M)$ and $\widetilde{\Gamma}_M/G_M$ are the images of geodesic quadrilaterals which are circles with index $2$, and $\widetilde{\Gamma}_M'/G_M^+$ and $\widetilde{\Gamma}_M/G_M^+$ are the red line and the blue line (or circle) as in Figure \ref{fig:GMadd}, which intersect the fibres of $S^3_2/G_M^+$ orthogonally. Note that $\widetilde{\Gamma}_M'/Isom^+(M)$ is the union of four edges which may give unknotted or knotted allowable edges in \cite{WWZZ}.

As the image of the fundamental quadrilateral, in $S^3_2/G_M$ the image of the blue line (or circle) bounds a minimal disk, and the $\pi$-rotation around the red line in Figure \ref{fig:GMadd} induces a symmetry on the minimal disk. This corresponds to the symmetry given before Proposition \ref{Pro:iso}. Similarly the $\pi$-rotation around the blue line (or circle) induces a symmetry on $S^3_2/Isom^+(M)$, which will change the edges $a$ and $a'$. This is exactly the $\mathbb{Z}_2$ action given in \cite{WWZ}.

Finally, the main result in \cite{WWZZ} relies on the classification result of spherical orbifolds in \cite{D1}, and it is not quite satisfied at two aspects:

1. It is not quite clear that what do the actions look like.

2. The computer is used in the proof.

The spherical orbifolds can be divided into two classes: the fibred orbifolds and the non-fibred orbifolds. For the non-fibred orbifolds, these two problems can be solved since in \cite{D2} Dunbar established the correspondence between the orbifolds and the quaternion representations of their fundamental groups. Moreover, explicit fundamental domains of the group actions are given. Now with the results in \cite{MS} and the explicit constructions in this paper, these two problems can also be solved for the fibred orbifolds. Namely, for maximal extendable actions in \cite{WWZZ} one can construct the surfaces in $S^3$ explicitly and represent the groups by the quaternion, and it is feasible to prove the main result in \cite{WWZZ} without the computer.

\bibliographystyle{amsalpha}

\end{document}